\documentclass[a4paper, oneside, 10pt, reqno]{amsart} \usepackage[english]{babel}
\usepackage{latexsym,amsfonts,amsmath,enumerate,graphics,enumerate,amsthm,amssymb}
\title{Entropy of group actions beyond uniform lattices}

\author{Till Hauser}
\address{T.H., Max Planck Institute for mathematics, 
53111 Bonn,
Germany} 
\email{hauser.math@mail.de}

\author{Friedrich Martin Schneider}
\address{F.M.S., Institute of Discrete Mathematics and Algebra, TU Bergakademie Freiberg, 09596 Freiberg, Germany}
\email{martin.schneider@math.tu-freiberg.de}

\let\epsilon=\varepsilon

\theoremstyle{definition}
\newtheorem{definition}{Definition}[section]

\theoremstyle{default}

\newtheorem{theorem}[definition]{Theorem}
\newtheorem{proposition}[definition]{Proposition}
\newtheorem{lemma}[definition]{Lemma}
\newtheorem{corollary}[definition]{Corollary}
\newtheorem*{introtheorem}{Theorem}

\newcommand{\Pcov}{\operatorname{P}}

\theoremstyle{remark}
\newtheorem{remark}[definition]{Remark}
\newtheorem{example}[definition]{Example}

\renewcommand{\epsilon}{\varepsilon}
\renewcommand{\phi}{\varphi}
\newcommand{\defeq}{\mathrel{\mathop:}=}

\DeclareMathOperator{\match}{match}
\DeclareMathOperator{\id}{id}
\DeclareMathOperator{\Sym}{Sym}

\DeclareMathOperator{\diam}{diam}

\makeatletter
\def\moverlay{\mathpalette\mov@rlay}
\def\mov@rlay#1#2{\leavevmode\vtop{%
		\baselineskip\z@skip \lineskiplimit-\maxdimen
		\ialign{\hfil$\m@th#1##$\hfil\cr#2\crcr}}}
\newcommand{\charfusion}[3][\mathord]{
	#1{\ifx#1\mathop\vphantom{#2}\fi
		\mathpalette\mov@rlay{#2\cr#3}
	}
	\ifx#1\mathop\expandafter\displaylimits\fi}
\makeatother

\newcommand{\bigcupdot}{\charfusion[\mathop]{\bigcup}{\boldsymbol{\cdot}}}

\usepackage{xcolor}

\begin{document}
\maketitle

\begin{abstract}
Entropy of measure preserving or continuous actions of amenable discrete
groups allows for various equivalent approaches. Among them are the ones
given by the techniques developed by Ollagnier and Pinchon on the one hand and the
Ornstein-Weiss lemma on the other. We extend these two approaches to the
context of actions of amenable topological groups. In contrast to the
discrete setting, our results reveal a remarkable difference between the
two concepts of entropy in the realm of non-discrete groups: while the
first quantity collapses to 0 in the non-discrete case, the second
yields a well-behaved invariant for amenable unimodular groups.
Concerning the latter, we moreover study the corresponding notion of
topological pressure, prove a Goodwyn-type theorem, and establish the
equivalence with the uniform lattice approach (for locally compact
groups admitting a uniform lattice). Our study elaborates on a version
of the Ornstein-Weiss lemma due to Gromov.
 \end{abstract}




\section{Introduction}

The concept of entropy for dynamical systems was introduced by the works
of A.\ N.\ Kolmogorov~\cite{kolmogorov1958new} and J.\ G.\
Sinai~\cite{sinai1959notion} and has since facilitated numerous
applications. The wealth of insights supported by this notion has
motivated a variety of useful generalizations to large classes of group
actions. Most notably, a powerful entropy theory has been developed both
for continuous and for measure preserving actions of amenable discrete
groups (see,
e.g.,~\cite{stepin1980variational,ollagnier1982variational,ornstein1987entropy,
ollagnier1985ergodic, ward1992abramov, lindenstrauss2000mean,
huang2011local, ceccherini2014analogue, Yan, YanandZeng}). More recent, very influential advances beyond the realm of amenable groups include \emph{sofic
entropy} (see~\cite{bowen2018brief} for a comprehensive account) and
\emph{naive entropy}~\cite{downarowicz2015shearer,
burton2017naivenetropy}. 
   
In many concrete situations 
 the acting group of a dynamical
system comes along equipped with a natural group topology that is
compatible with the action. Therefore, it seems natural to build on
amenability of topological groups (which is weaker than amenability of
the underlying discrete group) in order to develop a robust,
well-behaved entropy theory. For instance, in the study of aperiodic
order, it is natural to consider actions of non-discrete, metrizable,
$\sigma$-compact, locally compact (abelian) groups, such as
$\mathbb{R}^d$. In such case, one often uses a uniform lattice $\Lambda$
(for example $\mathbb{Z}^d$ in $\mathbb{R}^d$) and computes the entropy
of the original action as the entropy of the restricted action of the
discrete (abelian) group $\Lambda$. However, the existence of a uniform
lattice in a locally compact group is a strong requirement, and it is
natural to ask for an intrinsic definition of entropy for actions of
general (amenable) topological groups. In
\cite{Tagi-Zade} such an intrinsic definition is presented for actions
of $\mathbb{R}^d$, and for amenable unimodular locally compact groups further ideas are discussed in
\cite{ornstein1987entropy}. 
   
The purpose of the present work is to put forward entropy theory for
actions of (non-discrete) amenable topological groups. To this end, we
pursue and compare two approaches based on different approximation
techniques for amenable topological groups. While the first one combines
a topological version of F\o lner's amenability
criterion~\cite{schneider2016folner} with ideas of
Ollagnier and Pinchon~\cite{ollagnier1978unenouvelle,ollagnier1982variational,ollagnier1985ergodic}, the second one builds on the
concept of \emph{van Hove} nets and the quasi-tiling theory developed by
Ornstein and Weiss~\cite{ornstein1987entropy}.

	One of the main difficulties in defining entropy of non-discrete topological groups is that one wishes to consider a finite partition $\alpha$ and compute a common refinement \begin{displaymath}
		\alpha_A \, \defeq \, \bigvee\nolimits_{g\in A} \left.\! \left\{g^{-1}(A) \, \right\vert A\in \alpha\right\}
	\end{displaymath} for certain 'averaging' subsets $A\subseteq G$. 
	In the definition of entropy one requires that $\alpha_A$ is also a finite partition for which it is crucial that $A$ is finite. 
	In the context of discrete amenable groups that causes no problem since F\o lner nets can be used for averaging and consist of finite sets. 
However, in the general context of amenable unimodular locally compact groups $G$ one often uses certain nets of compact subsets for averaging. To be more precise, let $\theta$ denote a Haar measure on a unimodular locally compact group $G$. For compact subsets $K,A\subseteq G$, the \emph{$K$-boundary of $A$} is defined as $\partial_K A \defeq K\overline{A}\cap K\overline{G\setminus A}$. 
		A net $(A_i)_{i\in I}$ of compact subsets of $G$ is called \emph{van Hove} if $\theta(\partial_KA_i)/\theta(A_i)$ converges to $0$ for every compact subset $K\subseteq G$. 	
	For discrete amenable groups, the concepts of F\o lner and van Hove nets agree \cite{tempelman1984specific}. 
	Since van Hove sets in non-discrete amenable locally compact groups are not necessarily finite, the classical definition of entropy cannot be applied directly. 
	
	In \cite{schneider2016folner} a generalization of F\o lner nets to the context of general amenable topological groups is introduced that allows to average over finite sets. 
	To introduce this concept, let $G$ be a topological group. For finite subsets $E,F\subseteq G$ and any neighbourhood $U$ of the neutral element in $G$, we define the \emph{$U$-matching number} $\match_{\Rsh}(E,F,U)$ as 
	the maximum cardinality of a subset $M\subseteq E$ for which  there exists an injection $\iota\colon M\to F$ with $x\in U\iota(x)$ for all $x\in M$. A net $(F_i)_{i\in I}$ of finite, non-empty subsets of $G$ is called a \emph{thin F\o lner net} if, for every $g \in G$ and every identity neighborhood $U$ in $G$, the ratio $\match_{\Rsh}(F_i,gF_i,U)/|F_i|$ converges to $1$. According to~\cite[Theorem~4.5]{schneider2016folner}, the topological group $G$ is amenable if and only if $G$ admits a thin F\o lner net. Now, let $(X,\mu)$ be any probability space. As usual, for a finite measurable partition $\alpha$ of $X$, we consider the \emph{Shannon entropy} defined as \begin{displaymath}
		H_\mu(\alpha) \, \defeq \, -\sum\nolimits_{A\in \alpha \mid \mu(A)>0}\mu(A)\log(\mu(A)) .
	\end{displaymath} 
	Based on the preceding discussion, if $G$ is amenable, then one may define, for any measure preserving action of $G$ on $(X,\mu)$, the corresponding \emph{Ollagnier entropy} as 
	\[ \left. \operatorname{E}_\mu^{(Oll)}(\pi)\,\defeq\,\sup \left\{ \lim_{i\in I}\frac{H_\mu(\alpha_{F_i})}{|F_i|} \, \right\vert \alpha \text{ finite measurable partition of } X \right\}, \]
	where $(F_i)_{i\in I}$ is a thin F\o lner net. 
	See Subsection \ref{sub:actionsoftopologicalgroups}
 for the notion of a measure preserving action of a topological group.
	In Subsection \ref{sub:Ollagnierentropy} we will show that all the limits in this formula exist and are independent of the choice of a thin F\o lner net. 
	This will be achieved by generalizing a technique of \cite{ollagnier1985ergodic}. 
	To present this generalization, we topologize the set $\mathcal{F}(G)$ of all finite subsets of $G$ in a natural way. For details on this topology see Subsection \ref{sub:continuityofsetfunctions}.  
%
	We will show that, for every measure preserving action of $G$ on a probability space $(X,\mu)$, the induced map \begin{displaymath}
		\mathcal{F}(G) \, \longrightarrow \, \mathbb{R}_{\geq 0}, \quad F \, \longmapsto \, H_\mu(\alpha_F)
	\end{displaymath} is continuous, right-invariant, and {satisfies Shearer's inequality} (see Subsection~\ref{subsection:combinatorics} for the relevant definitions). This observation will be combined with the following abstract convergence theorem, where $\mathcal{F}_{+}(G) \defeq \mathcal{F}(G)\setminus \{ \emptyset \}$. 
\begin{introtheorem}[Topological version of Ollagnier's lemma; Theorem \ref{the:ollagnierslemma}]
	Let $G$ be an amenable topological group, and let $f\colon \mathcal{F}(G)\to \mathbb{R}_{\geq 0}$ be a continuous, right-invariant function satisfying Shearer's inequality. Then, for any thin F\o lner net $(F_i)_{i\in I}$ in $G$, the following limit exists and satisfies
	\[\lim_{i\in I}\frac{f(F_i)}{|F_i|} \, = \, \inf_{F\in \mathcal{F}_{+}(G)}\frac{f(F)}{|F|}. \]
\end{introtheorem}
	Now recall that for a measure preserving action $\pi$ of a (not necessarily amenable) topological group $G$ on a probability space $(X,\mu)$ we define the \emph{naive entropy} entropy as 
	\[\operatorname{E}^{(nv)}_\mu(\pi) \defeq \sup_{\alpha}\inf_{F\in \mathcal{F}_{+}(G)}\frac{H_\mu(\alpha_F)}{|F|},\]
	where the supremum is taken over all finite measurable partitions $\alpha$ of $X$. 
	Similarly to the well known context of actions of discrete amenable groups, we also obtain from Ollagnier's lemma that for measure preserving actions of general amenable groups we have
	\begin{displaymath}
		\operatorname{E}^{(Oll)}_\mu(\pi) \, = \, \operatorname{E}^{(nv)}_\mu(\pi) .
	\end{displaymath} From the following theorem we thus observe that 
	$\operatorname{E}^{(Oll)}_\mu(\pi)=0$ for all measure preserving actions of non-discrete amenable groups that are metrizable or locally compact. 
	
	\begin{introtheorem}
	Let $G$ be a non-discrete topological group that is metrizable or locally compact. 
		If $\pi$ is a measure preserving action of $G$ on a probability space $(X,\mu)$, then
		$\operatorname{E}_\mu^{(nv)}(\pi)=0$.
	\end{introtheorem}
	We achieve this result by finding an appropriate topological model for $\pi$. We then prove a similar result for naive topological entropy and use this to show that naive topological entropy bounds naive entropy (see Corollary~\ref{cor:naiveentropyiszero} below for further details). We note that the phase space is not assumed to be a Lebesgue space and that our notion of topological model differs from the standard notion. 
	
	Unfortunately, this shows that the approach via thin F\o lner nets does not yield a useful invariant in the context beyond discrete groups and one needs to come back to the averaging along van Hove nets. A natural idea in this context is to reconsider uniform lattices and to note that a uniform lattice is a Delone subgroup. Here we call a subset $\omega\subseteq G$ \emph{Delone} if $\omega$ is \emph{uniformly discrete} (i.e., there exists a neighbourhood $V$ of the identity in $G$ such that $(gV)_{g \in \omega}$ is a disjoint family) and \emph{relatively dense} (i.e., there exists a compact subset $K\subseteq G$ such that $K\omega=G$).
	Since any locally compact topological group contains Delone sets, it is natural to ask whether one can avoid the group structure in $\omega$ while defining entropy.
	As a net of averaging subsets of $\omega$ we consider $(A_i\cap \omega)_{i\in I}$, where $(A_i)_{i\in I}$ is a van Hove net in $G$. 
	For a measure preserving action $\pi$ of an amenable unimodular locally compact group $G$ on a probability space $(X,\mu)$, we define the \emph{Ornstein-Weiss entropy} as
		\[\operatorname{E}_\mu^{(OW)}(\pi) \defeq \sup_{\alpha}\limsup_{i\in I}\frac{H_\mu(\alpha_{A_i\cap \omega})}{\theta(A_i)},\]
	where the supremum is taken over all finite measurable partitions $\alpha$ of $X$. 
	Note that, if $\omega$ is a uniform lattice in $G$ and $(A_i)_{i\in I}$ is a van Hove net, then $(A_i\cap \omega)_{i\in I}$ constitutes a F\o lner net in $\omega$, and $\operatorname{dens}(\omega) \defeq \lim_{i\in I}|A_i\cap \omega|/\theta(A_i)$ exists and is independent of the choice of a van Hove net (see Lemma \ref{lem:vanHovenetinlattice}). 
	We thus obtain directly from our definition 
	that the Ornstein-Weiss entropy coincides with the usual notion of entropy whenever $G$ contains a uniform lattice. 
	It is natural to ask whether our definition depends on the choice of a Delone set and the choice of a van Hove net. 
	\begin{introtheorem}
		The Ornstein-Weiss entropy of a measure preserving action of an amenable unimodular locally compact group on a probability space is independent of both the choice of a van Hove net and the choice of a Delone set. 
	\end{introtheorem}
	We achieve this independence in Section \ref{sec:entropyandpressure} as follows. As in the context of naive entropy we first find a topological model for our action. With the additional topological structure we then present an alternative approach to Ornstein-Weiss entropy, which does not involve Delone sets.  	
	In order to show independence of the choice of a van Hove net we will apply the following topological version of the Ornstein-Weiss lemma to this alternative approach.
	To state our version of the Ornstein-Weiss lemma, let us denote by $\mathcal{K}(G)$ the set of all compact subsets of $G$ (see Subsection~\ref{subsection:combinatorics} for the relevant definitions concerning real-valued functions on $\mathcal{K}(G)$). 
	\begin{introtheorem}[Ornstein-Weiss lemma; Theorem \ref{the:Ornstein-Weisslemma}]
		Let $\theta$ be a Haar measure on an amenable unimodular locally compact group $G$, and let $f\colon \mathcal{K}(G)\to \mathbb{R}$ be monotone, right-invariant, and subadditive. Then, for every van Hove net $(A_i)_{i\in I}$ in $G$, the limit
		\begin{displaymath}
				\lim_{i\in I}\frac{f(A_i)}{\theta(A_i)}
		\end{displaymath}
		 exists, is finite, and does not depend on the particular choice of a van Hove net $(A_i)_{i\in I}$.
	\end{introtheorem}
	The Ornstein-Weiss lemma is based on the quasi-tiling techniques developed by D.\ Ornstein and B.\ Weiss in \cite{ornstein1987entropy}. In \cite{gromov1999topological} M.\ Gromov gives a brief sketch of a proof, which is worked out in detail in  \cite{krieger2007lemme, krieger2010ornstein, huang2011local, ceccherini2014analogue}. A proof of the Ornstein-Weiss lemma is also contained in the appendix of \cite{lindenstrauss2000mean}.
	Here we do not intend to list all important references on this behalf and recommend  \cite{hauser2020relative} for further historical comments. 
	However, note that none of these proofs (aside from \cite{ornstein1987entropy} and the brief sketch in \cite{gromov1999topological}) considers non-discrete topological groups. 	
	In \cite{hauser2020relative} this result was then extended to certain non-discrete topological groups studied in the context of aperiodic order but it remained open, whether the result holds for all amenable unimodular locally compact groups. 
	For a proof of the full statement, we show and use a slight improvement (Proposition \ref{pro:quasitiling}) of the quasi-tiling results  of \cite{ornstein1987entropy} in Section \ref{sec:OWlemma}. We note that the claim in \cite{gromov1999topological} is even stronger than our version: we did not succeed in avoiding the assumption of monotonicity in our proof, and it remains open whether this is possible.  
	
	As another application of these ideas we will furthermore define and study the \emph{Ornstein-Weiss topological pressure} of a continuous action $\pi$ of an amenable unimodular locally compact group $G$ on a compact Hausdorff space $X$ and $f\in C(X)$ as follows. 
	For a finite open cover $\mathcal{U}$ and a compact subset $A\subseteq G$ we denote 
	\begin{displaymath}
		P_f(\mathcal{U}) \, \defeq \, \log\left(\sum\nolimits_{U\in \mathcal{U}}\sup\nolimits_{x\in U}e^{f(x)}\right)
	\end{displaymath} and
	$f_{A}(x) \defeq \int_{A}f(g.x)d\theta(g)$. 
	For a finite subset $F\subseteq G$ we furthermore consider the common refinement $\mathcal{U}_F \defeq \bigvee_{g\in F}\{g^{-1}(U) \mid U\in \mathcal{U}\}$. 
	We define the \emph{Ornstein-Weiss topological pressure} of $\pi$ with respect to $f$ as 
	\[\operatorname{p}_f^{(OW)}(\pi) \, \defeq \, \sup_{\mathcal{U}}\limsup_{i\in I}\frac{P_{f_{A_i}}(\mathcal{U}_{A_i\cap \omega})}{\theta(A_i)},\]
	where the supremum is taken over all finite open covers $\mathcal{U}$ of $X$.
	Again we will present a different approach to this notion that does not require the usage of Delone sets and apply the Ornstein-Weiss lemma to it to obtain the following. 
	\begin{introtheorem}
	The Ornstein-Weiss topological pressure of a continuous action of an amenable unimodular locally compact group on a compact Hausdorff space is independent of the choice of a van Hove net and the choice of a Delone set. 
	\end{introtheorem}
	We obtain the following version of Goodwyn's theorem. See Theorem \ref{the:OW:Goodwyns} for details. It remains open whether the variational principle holds in this context. 
	
	\begin{introtheorem}
		Let $\pi$ be a continuous action of an amenable unimodular locally compact group $G$ on a compact Hausdorff space $X$. If $\mu$ is a $\pi$-invariant regular Borel probability measure on $X$ and $f\in C(X)$, then \begin{displaymath}
				\operatorname{E}_\mu^{(OW)}(\pi)+\mu(f) \, \leq \, \operatorname{p}_f^{(OW)}(\pi) .
		\end{displaymath} 
	\end{introtheorem}

%
%
%
%
%
%

\section{Preliminaries}

\subsection{General notation}

Let $X$ be a set. The symmetric difference of two subsets $A,B\subseteq X$ will be denoted by $A \triangle B \defeq (A\setminus B)\cup (B\setminus A)$. The cardinality of a finite subset $F\subseteq X$ is denoted by $|F|$. We denote by $Y^X$ the set of all maps from $X$ to another set $Y$. 
We denote by $\mathcal{F}(X)$ the set of all finite subsets of $X$. Furthermore, let $\mathcal{F}_{+}(X) \defeq \mathcal{F}(X)\setminus \{ \emptyset \}$. For finite sets $\mathcal{U}$ and $\mathcal{V}$ of subsets of $X$, we say that $\mathcal{U}$ is \emph{finer than} $\mathcal{V}$ if for each $U\in \mathcal{U}$ there exists $V\in \mathcal{V}$ such that $U\subseteq V$. 
We also write $\mathcal{V}\preceq \mathcal{U}$ in this situation. For two (indexed) families $(A_{i})_{i \in I}$ and $(B_{j})_{j \in J}$ of sets with $I \cap J = \emptyset$, we define $(A_{i})_{i \in I} \sqcup (B_{j})_{j \in J} \defeq (C_{k})_{k \in I \cup J}$ by \begin{displaymath}
	C_{k} \, \defeq \, \begin{cases}
							\, A_{k} & \text{if } k \in I, \\
							\, B_{k} & \text{if } k \in J
						\end{cases} \qquad (k \in I \cup J) .
\end{displaymath} Moreover, if $(A_{i})_{i \in I}$ is a family of sets and $A$ is a set, then we let \begin{displaymath}
	(A_{i})_{i \in I} \sqcup (A) \, \defeq \, (A_{i})_{i \in I} \sqcup (A)_{J} ,
\end{displaymath} where $J$ is a singleton set disjoint from $I$. Now, let $X$ be a topological space. Then we denote by $\mathcal{K}(X)$ the set of all compact subsets of $X$. 
The closure of a subset $A\subseteq X$ is denoted by $\overline{A}$ and we write $\partial A$ for the topological boundary of $A$ in $X$. Finally, let $G$ be a group. For any two subsets $A,B\subseteq G$, we define $AB \defeq \{ab \mid a\in A, \, b\in B\}$. For $A \subseteq G$ and $g \in G$, we let $gA \defeq \{ g \}A$ and $Ag \defeq A\{ g \}$.
	

\subsection{Compact Hausdorff spaces}

Let $X$ be a compact Hausdorff space. Let $C(X)$ denote the space of all
continuous functions on $X$, and let $\mathbb{U}_X$ denote the
\emph{uniformity} of $X$, i.e., the set of all neighbourhoods (in
$X\times X$) of the diagonal $\{(x,x) \mid x\in X\}$. The members of
$\mathbb{U}_X$ are called \emph{entourages} of $X$. Let $\eta \in
\mathbb{U}_X$. Then $\eta \in \mathbb{U}_X$ is said to be
\emph{symmetric} if $\eta=\{(y,x) \mid (x,y)\in \eta\}$, and $\eta$ is
called \emph{open} if $\eta$ is open with respect to the product
topology on $X\times X$. Furthermore, let \begin{displaymath}
        \eta \eta \, \defeq \, \{ (x,z) \in X \times X \mid \exists y
\in X \colon \, (x,y), (y,z) \in \eta \} .
\end{displaymath} For any $x\in X$, we define $\eta[x] \defeq \{y\in X
\mid (y,x)\in \eta\}$. A set $\mathcal{F}$ of subsets of $X$ is said to
be \emph{at scale $\eta$} if each $M\in \mathcal{F}$ satisfies
$M^2\subseteq \eta$. Given an open cover $\mathcal{U}$ of $X$, a
symmetric entourage $\kappa \in \mathbb{U}_X$ is said to be
\emph{Lebesgue} with respect to $\mathcal{U}$ if for each $x\in X$ there
exists $U\in \mathcal{U}$ such that $\kappa[x]\subseteq U$. A standard
argument shows that any open cover of a compact Hausdorff space admits
an open Lebesgue entourage. For further details on uniformities and, in
particular, the notion of \emph{uniform continuity}, the reader is
referred to~\cite{Kelley}.


\subsection{Topological groups}

Let $G$ be a topological group, i.e., a group $G$ equipped with a Hausdorff topology such that the mapping $G\times G \to G, \, (g,h)\mapsto gh^{-1}$ is continuous. The neighborhood filter of the neutral element in $G$ will be denoted by $\mathcal{U}(G)$. 
 As usual, a subset $A \subseteq G$ of a topological group $G$ will be called \emph{precompact} if for any $U \in \mathcal{U}(G)$ there exists a finite subset $F\subseteq G$ such that $A \, \subseteq \, UF$. Note that precompact subsets of a locally compact topological group are characterized by having a compact closure. 

\subsection{Actions of topological groups}
\label{sub:actionsoftopologicalgroups}

For a start, let $G$ be a group and let $X$ be a set. A map $\pi\colon G\to X\times X$ is called an \emph{action (of $G$ on $X$)} if \begin{itemize}
	\item[---\,] for every $g \in G$, the map $\pi^{g} \colon X \to X, \, x \mapsto \pi(g,x)$ is a bijection, and
	\item[---\,] the resulting map $G \to \Sym (X), \, g \mapsto \pi^{g}$ to the full symmetric group $\Sym (X)$, i.e., the group of all bijections from $X$ to $X$, is a group homomorphism. 
\end{itemize} As is customary, if the action $\pi$ is clear from the context, then we also write $g.x \defeq \pi(g,x)$ and $g.A \defeq\pi^{g}(A)$ for $g\in G$, $x\in X$ and $A\subseteq X$.

Now, let $G$ be a topological group. An action $\pi$ of $G$ on a topological space $X$ is called \emph{continuous} if $\pi\colon G\times X\to X$ is continuous with respect to the product topology on $G\times X$. Given a continuous action $\pi$ of $G$ on a compact Hausdorff space $X$, we call a regular Borel probability measure $\mu$ on $X$ \emph{$\pi$-invariant} (or simply \emph{invariant}) if $\mu(g.A) = \mu(A)$ for every $g \in G$ and every measurable subset $A \subseteq X$. Furthermore, an action $\pi$ of $G$ on a probability space $(X,\mu)$ is called \emph{measure preserving} if \begin{itemize}
	\item[$(1)$] $\pi^g \colon X \to X$ is measurable for every $g \in G$,
	\item[$(2)$] $\mu(A)=\mu(g.A)$ for every $g \in G$ and every measurable $A\subseteq X$, and	
	\item[$(3)$] for every measurable subset $A \subseteq X$, the map \begin{displaymath}
	G \, \longrightarrow \, \mathbb{R}_{\geq 0} , \quad g \, \longmapsto \, \mu(A\triangle g.A)
\end{displaymath} is continuous, which is equivalent, by $\pi$ being measure preserving, to requiring that \begin{displaymath}
	\mu(A\triangle g.A) \, \longrightarrow \ 0 \, \text{ as } \, g\to e_G .
\end{displaymath}
\end{itemize}
	Note that (3) is trivially satisfied whenever $G$ is discrete. 
	More generally, if $\pi$ is an action of a locally compact group $G$ that satisfies (1) and (2), then (3) is implied by $\pi\colon G\times X\rightarrow X$ being measurable with respect to the product $\sigma$-algebra on $G\times X$. For a proof of the latter see \cite[Chapter II.\S1, Lemma~1]{ornstein1987entropy}.



\begin{lemma}\label{lem:continuoslymeasurepres}
	Any continuous action $\pi$ on a compact Hausdorff space $X$ equipped with an invariant and regular Borel probability measure $\mu$ is measure preserving. 
\end{lemma}
\begin{proof}
	Consider a Borel measurable subset $A\subseteq X$ and let $\epsilon>0$. By regularity of $\mu$ and Urysohn's lemma, there exists $f\in C(X)$ such that 
	$\Vert\chi_A-f\Vert_1\leq \tfrac{\epsilon}{3}$.
	Thanks to a standard argument (see, e.g.,~\cite[Theorem 4.17]{eisner2015operator}), we find a neighborhood $U$ of the neutral element in $G$ such that 
	$\left\Vert f-f\circ \pi^g\right\Vert_\infty\leq \tfrac{\epsilon}{3}$ for every $g \in U$. 
	In particular, if $g \in U$, then
	\begin{align*}
		\mu(A\triangle (g^{-1}).A)
		&=\|\chi_A-\chi_{(g^{-1}).A}\|_1=\left\|\chi_A-\chi_A\circ \pi^g\right\|_1\\
		&\leq \left\|\chi_A-f\right\|_1
		+
		\left\|f-f\circ \pi^g\right\|_\infty
		+
		\left\|f \circ \pi^g-\chi_A \circ \pi^g\right\|_1
		\leq \epsilon. \qedhere
	\end{align*}
\end{proof}

\subsection{Amenability}\label{subsection:amenability}

We continue with a brief review of the concept of amenability for general topological groups. To this end, let $G$ be a topological group. Then $G$ is called \emph{amenable} if every continuous action of $G$ on a non-empty compact Hausdorff space admits an invariant regular Borel probability measure. Equivalently, $G$ is amenable if and only if every continuous action of $G$ by affine homeomorphisms on a non-void compact convex subset of a locally convex topological vector space has a fixed point. A well-known result by Rickert~\cite[Theorem~4.2]{rickert} asserts that $G$ is amenable if and only if there exists a left-invariant mean on the space of right-uniformly continuous bounded real-valued functions on $G$. For more details on amenability of general topological groups, we refer to~\cite{PestovBook,GrigorchukDeLaHarpe}.

For the purposes of the present work, a characterization of amenability of topological groups in terms of almost invariant finite non-empty subsets (Theorem~\ref{theorem:thin.folner} below) will be relevant. This result is due to~\cite{schneider2016folner} and constitutes a topological version of F\o lner's amenability criterion~\cite{folner1956groups}. To recall the precise statement, we will clarify some convenient notation concerning matchings in bipartite graphs, essentially following~\cite{schneider2016folner}. Let us consider a \emph{bipartite graph}, that is, a triple $\mathcal{B} = (E,F,R)$ consisting of two finite sets $E$ and $F$ and some subset $R \subseteq E \times F$. A \emph{matching} in $\mathcal{B}$ is an injective mapping $\phi \colon D \to F$ such that $D \subseteq E$ and $(x,\phi(x)) \in R$ for all $x \in D$. The \emph{matching number} of $\mathcal{B}$ is defined as \begin{displaymath}
	\match (\mathcal{B}) \defeq \sup \{ \vert D \vert \mid \phi \colon D \to F \textnormal{ matching in } \mathcal{B} \} .
\end{displaymath} With regard to amenability of topological groups, the following type of bipartite graphs and matching numbers has been revealed as relevant in~\cite{schneider2016folner}.

\begin{definition} Let $G$ be a topological group. For $U \in \mathcal{U}(G)$, let us define \begin{displaymath}
	U_{\Rsh} \, \defeq \, \left\{ (x,y) \in G \times G \left\vert \, xy^{-1} \in U \right\} \! . \right.
\end{displaymath} For $E,F \in \mathcal{F}(G)$ and $U \in \mathcal{U}(G)$, we define \begin{displaymath}
	\match_{\Rsh}(E,F,U) \, \defeq \, \match \! \left( E,F, U_{\Rsh} \cap (E \times F) \right) .
\end{displaymath} A net $(F_{i})_{i \in I}$ of non-empty finite subsets of $G$ will be called a \emph{thin F\o lner net} in $G$ if, for every $U \in \mathcal{U}(G)$ and for every $g \in G$, \begin{displaymath}
	\frac{\match_{\Rsh}(F_{i},gF_{i},U)}{\vert F_i \vert} \, \overset{i\in I}{\longrightarrow} \, 1.
\end{displaymath} \end{definition}

\begin{theorem}[\cite{schneider2016folner}, Theorem~4.5]\label{theorem:thin.folner} A topological group is amenable if and only if it admits a thin F\o lner net. \end{theorem}

For later use, let us include the following reformulation of the thin F\o lner property.

\begin{lemma}\label{lemma:folner.permutations} Let $G$ be a topological group. For $F \in \mathcal{F}(G)$, $U \in \mathcal{U}(G)$ and $g \in G$, \begin{displaymath}
		\match_{\Rsh} (F,gF,U) \, = \, \sup\nolimits_{\gamma \in \mathrm{Sym}(F)} \vert \{ x \in F \mid (\gamma(x),gx) \in U_{\Rsh} \} \vert  .
\end{displaymath} \end{lemma}

\begin{proof} ($\leq$) Let $\phi \colon F \to gF$ be a bijection and consider the set $D \defeq \{ x \in F \mid (x,\phi (x)) \in U_{\Rsh} \}$. Of course, $\gamma \colon F \to F, \, x \mapsto \phi^{-1}(gx)$ is a bijection. Note that $\phi(x) = g\gamma^{-1}(x)$ for each $x \in F$. Let $C \defeq \{ x \in F \mid (\gamma (x), gx) \in U_{\Rsh} \}$. For every $x \in F$, we have \begin{align*}
		\gamma^{-1}(x) \in C \ &\Longleftrightarrow \ \left(\gamma \!\left(\gamma^{-1}(x)\right),g\gamma^{-1}(x)\right) \in U_{\Rsh} \ \Longleftrightarrow \ \left(x, g\gamma^{-1}(x)\right) \in U_{\Rsh} \\
		& \Longleftrightarrow \ (x,\phi (x)) \in U_{\Rsh} \ \Longleftrightarrow \ x \in D ,
	\end{align*} which means that $\gamma^{-1}(D) = C$, which in turn implies that $\vert C \vert = \vert D \vert$.
	
	($\geq$) Let $\gamma \in \mathrm{Sym}(F)$ and define $C \defeq \{ x \in F \mid (\gamma(x),gx) \in U_{\Rsh} \}$. We consider the bijection $\phi \colon F \to gF, \, x \mapsto g\gamma^{-1}(x)$ and let $D \defeq \{ x \in F \mid (x,\phi(x)) \in U_{\Rsh} \}$. For every $x \in F$, \begin{align*}
		\gamma(x) \in D \ &\Longleftrightarrow \ (\gamma (x),\phi(\gamma (x))) \in U_{\Rsh} \ \Longleftrightarrow \ (\gamma(x),gx) \in U_{\Rsh} \ \Longleftrightarrow \ x \in C ,
	\end{align*} which means that $\gamma(C) = D$. Hence, $\vert C \vert = \vert D \vert$. \end{proof}

\subsection{Geometric notions}
We continue with a few bits of geometric terminology. Again, let $G$ be a topological group.
A subset $\omega\subseteq G$ is called 
\emph{uniformly discrete} if there exists an identity neighborhood $U \in \mathcal{U}(G)$ such that $(gU)_{g \in \omega}$ is a disjoint family, in which case the set $\omega$ will also be called \emph{$U$-discrete}. 
A subset $\omega \subseteq G$ is called \begin{itemize}
	\item[---]	\emph{relatively dense} if there exists a compact subset $K\subseteq G$ such that $K\omega=G$,
	\item[---]  \emph{Delone} if $\omega$ is uniformly discrete and relatively dense,
	\item[---]	\emph{locally finite} if $\omega\cap K$ is finite for every compact $K\subseteq G$.
\end{itemize} Note that any Delone set in a locally compact topological group is locally finite.  

\subsection{van Hove nets}	 

	Let $G$ be a locally compact topological group. 
	Recall that $G$ admits a \emph{(left) Haar measure}, that is, a non-zero, regular, locally finite Borel measure $\theta$ on $G$ which satisfies $\theta(A)=\theta(gA)$ for all Borel measurable $A\subseteq G$ and all $g\in G$.
	Haar measure is unique up to scaling, i.e., for two Haar measures $\theta$ and $\theta'$ on $G$ there exists a constant $c$ such that $\theta'=c\theta$.	
	As usual, $G$ is called \emph{unimodular} if some (every) Haar measure $\theta$ on $G$ furthermore satisfies $\theta(Ag)=\theta(A)$ for all Borel measurable $A\subseteq G$ and all $g\in G$.
	
	Throughout the manuscript, whenever $G$ is assumed to be a locally compact group, we let $\theta$ denote a fixed Haar measure on $G$. Now, let $G$ be a unimodular locally compact group. For any measurable subset $A\subseteq G$ we have $\theta(A)=\theta(A^{-1})$. 
	We define 
	$\partial_K A \defeq K\overline{A}\cap K\overline{G\setminus A}$ 
	for subsets $K,A\subseteq G$.
	Note that $\partial_K A$ is compact and hence measurable, whenever $K$ is compact and $A$ is precompact. 
	For $\epsilon>0$ and a compact subset $K\subseteq G$ we say that a precompact and measurable subset $A\subseteq G$ of positive Haar measure is 
	\emph{$(\epsilon,K)$-invariant} if 
	\[\alpha(A,K)\defeq\frac{\theta(\partial_K A)}{\theta(A)}<\epsilon.\]	
	Recall from the introduction that a net $(A_i)_{i\in I}$ of compact subsets of $G$ is called \emph{van Hove} if $\lim_{i\in I}\alpha(A_i,K)=0$ for every compact $K\subseteq G$ (where we assume implicitly that $\theta(A_i)>0$ for sufficiently large $i\in I$). A unimodular locally compact group is amenable if and only if it admits a van Hove net \cite{tempelman1984specific, hauser2020relative}.
	With the arguments from \cite[Subsection 2.6]{hauser2020Anote} it is straightforward to obtain the following. We include some comments on the proof for the convenience of the reader.

		\begin{lemma}\label{lem:pre:vanHovenets}
	Let $G$ be a unimodular locally compact group and let $\theta$ be a Haar measure on~$G$. Let $(A_i)_{i\in I}$ be a van Hove net in $G$ and let $K\subseteq G$ be non-empty and compact.
	\begin{itemize}
	\item[(i)] $(KA_i)_{i\in I}$ is a van Hove net which satisfies 
	$\lim_{i\in I}\theta(KA_i)/\theta(A_i)=1$. 
	\item[(ii)] There exists a van Hove net $(B_i)_{i\in I}$ (with the same index set) for which we have $\lim_{i\in I}\theta(B_i)/\theta(A_i)=1$ and which satisfies $KB_i\subseteq A_i$ and $B_i^c\subseteq K^{-1}A_i^c$ for all $i\in I$. 
	\end{itemize}			 
		\end{lemma}
\begin{proof}
	It is straightforward to show (i) \cite[Proposition 2.3]{hauser2020relative}. In order to see (ii) we set 
	$B_i:=\overline{\{g\in G;\, Kg\subseteq A_i\}}$ and follow well known arguments as in \cite[Proposition 2.2]{hauser2020Anote} in order to show $\lim_{i\in I}\theta(B_i)/\theta(A_i)=1$ and $KB_i\subseteq A_i$. Furthermore, for $g\in B_i^c$ we have $Kg\cap A_i^c\neq \emptyset$ and hence $g\in K^{-1}A_i^c$. 
\end{proof}

\subsection{Uniform lattices and van Hove nets}
	Let $G$ be a locally compact group. A discrete subgroup $\omega$ of $G$ is called a \emph{uniform lattice} if the quotient $G/ \omega$ is compact, i.e., $\omega$ is a Delone subgroup of $G$. A subset $C\subseteq G$ is said to be \emph{regular} if $\theta(\partial C)=0$ for some (any) left Haar measure $\theta$ on $G$.
	A fundamental domain of a uniform lattice $\omega$ is a Borel measurable set $C\subseteq G$ such that for any $g\in G$ there exist unique $c\in C$ and $v\in \omega$ such that $g=cv$.
	\begin{lemma}
		Any uniform lattice $\omega$ in a unimodular locally compact group $G$ admits a regular and precompact fundamental domain with non-empty interior.  
	\end{lemma}
	\begin{proof}
	We first show that $G$ admits a compact, regular neighbourhood of $e_G$ which satisfies $M^{-1}M\cap \omega=\{e_G\}$. Choose a compact neighbourhood $\hat{M}$ of $e_G$ which satisfies $\hat{M}^{-1}\hat{M}\cap \omega=\{e_G\}$. 
	From \cite[Chapter 6]{Kelley} we know that there exist finitely many continuous pseudometrics $d_1,\dots,d_J$ on $\hat{M}$ and $\epsilon>0$ such that \begin{displaymath}
		\eta \, \defeq \, \bigcap\nolimits_{j=1}^J \! \left. \left\{(x,y)\in \hat{M}^2 \, \right\vert d_j(x,y)<\epsilon \right\} \, \in \, \mathbb{U}_{\hat{M}}
	\end{displaymath} and such that $\eta[e_G]$ is contained in the interior of $\hat{M}$. 
	By a standard argument involving Froda's theorem, for each $j\in \{1,\dots, J\}$ there exists $r_j\in (0,\epsilon)$ such that the closed $d_j$-ball $\overline{B}_{r_j}^{d_j}(e_G)$ is regular. Then $M:=\bigcap_{j=1}^J\overline{B}_{r_j}^{d_j}(e_G)$ is a compact and regular neighbourhood of $e_G$ which satisfies $M^{-1}M\cap \omega=\{e_G\}$.
		
	Let $K\subseteq G$ be a compact subset such that $K\omega=G$. Then $\{gU \mid g\in K\}$ is an open cover of $K$ and thus there exists a finite sequence $(g_n)_{n=1}^N$ such that $\bigcup_{n=1}^N g_nU \supseteq K$. 
	From \cite{morris2015introduction, hauser2020relative} we know that 
	$C\defeq\bigcup_{n=1}^N g_nM \setminus \left(\bigcup_{i=1}^{n-1} g_iM\omega\right)$ is a precompact fundamental domain. Setting $F\defeq\omega\cap (\bigcup_{n=1}^N g_nM)^{-1}(\bigcup_{n=1}^N g_nM)$ we obtain a finite set with 
	$C\defeq\bigcup_{n=1}^N g_nM \setminus \left(\bigcup_{i=1}^{n-1} g_iMF\right)$ and 
	the regularity of $M$ implies $C$ to be regular. Since $M$ has non-empty interior, so does $C$. 
\end{proof}

		

\subsection{Combinatorial properties of set functions}\label{subsection:combinatorics}

\begin{definition} Let $X$ be a set. 
For $k\in \mathbb{N}_{>0}$, a \emph{$k$-cover} of $X$ is a finite family $(C_i)_{i\in I}$ of subsets of $X$ such that \begin{displaymath}
	\forall x \in X \colon \qquad \vert \{i\in I\mid x\in C_i\} \vert \, \geq \, k .
\end{displaymath}
Let $\mathcal{F} \subseteq \mathcal{P}(X)$ be closed under finite unions and binary intersections. A map $f \colon \mathcal{F} \to \mathbb{R}_{\geq 0}$ \begin{itemize}
		\item[---\,] is called \emph{monotone} if $f(E) \leq f(F)$ for all $E, F \in \mathcal{F}$ with $E \subseteq F$,
		\item[---\,] is called \emph{subadditive} if $f(\emptyset) = 0$ and $f(E \cup F) \leq f(E) + f(F)$ for all $E,F \in \mathcal{F}$,
		\item[---\,] \emph{satisfies Shearer's inequality} if, for every $k \in \mathbb{N}_{>0}$, every $F\in \mathcal{F}$ and every $k$-cover $(C_i)_{i\in I}$ of $F$ by members of $\mathcal{F}$, we have $f(F) \leq \frac{1}{k} \sum\nolimits_{i\in I} f(C_{i}),$
		\item[---\,] is called \emph{strongly subadditive} if $f(\emptyset) = 0$ and $f(E \cup F) \leq f(E) + f(F) - f(E \cap F)$ for all $E,F \in \mathcal{F}$.
\end{itemize} \end{definition}

The concept of covering multiplicity originates in the work of \cite{kelley1959measures}. 
The notion of satisfaction of Shearer's inequality stated above is a priori weaker that the one defined in~\cite{downarowicz2015shearer}, but is easily seen to be equivalent for monotone set functions. Also, it is shown in~\cite[Proposition~2.4]{downarowicz2015shearer} that, for monotone set functions, strong subadditivity implies Shearer's inequality, which in turn implies subadditivity.

\section{The Ollagnier lemma}	
\label{sec:Ollagnierlemma}

In this section we will be concerned with a topological version of Ollagnier's ergodic theorem \cite[Proposition 3.1.9]{ollagnier1985ergodic} and with topological and measure-theoretic entropy of continuous action of amenable topological groups. For a start, we briefly agree on some terminology concerning continuity of real-valued set functions.

\subsection{Continuity of set functions}
\label{sub:continuityofsetfunctions}

We are going to address some continuity matters for set functions. More
precisely, we will endow that set of all finite subsets of a topological
group with a suitable topology. For this purpose, let $G$ be a
topological group. For any finite subset $F \subseteq G$ and an identity
neighborhood $U \in \mathcal{U}(G)$, we define \begin{displaymath}
    U[F] \, \defeq \, \{ F' \in \mathcal{F}(G) \mid \exists \phi \colon
F' \to F \text{ bijective } \forall x \in F' \colon \, x\phi(x)^{-1} \in
U \} .
\end{displaymath} It is straightforward to verify that \begin{displaymath}
    \{ T \subseteq \mathcal{F}(G) \mid \forall F \in T \, \exists U \in
\mathcal{U}(G) \colon \, U[F] \subseteq T \} .
\end{displaymath} constitutes a topology on $\mathcal{F}(G)$. Moreover,
for each $F \in \mathcal{F}(G)$, the collection \begin{displaymath}
    \{ U[F] \mid U \in \mathcal{U}(G) \}
\end{displaymath} is a neighborhood basis of $F$ in the resulting
topological space $\mathcal{F}(G)$. Henceforth, whenever the set of all
finite subsets of a topological group is considered as a topological
space, any statement will be referring to this topology. In connection
with the concept of set functions discussed in
Subsection~\ref{subsection:combinatorics}, we note that an important
family of continuous set functions on topological groups is provided by
Lemma~\ref{lem:continuouslympandEntropy}.

For later use, we record some additional observations concerning the topology introduced above and the F\o lner-type matching conditions for amenability discussed in Subsection~\ref{subsection:amenability}.

\begin{lemma}\label{lemma:approximation.by.moving.sets} Let $G$ be any non-discrete topological group. If $U \in \mathcal{U}(G)$ and $E,F \in \mathcal{F}(G)$, then there exists $F' \in U[F]$ such that \begin{displaymath}
		\forall g \in E \setminus \{ e_G \} \colon \quad F' \cap gF' \, = \, \emptyset .
\end{displaymath} \end{lemma}

\begin{proof} Denote by $\mathcal{Z}$ the set of all pairs $(F',\phi)$ consisting of a finite subset $F' \subseteq G$ and an injective map $\phi \colon F' \to F$ such that $x\phi (x)^{-1} \in U$ for all $x \in F'$ and $F' \cap gF' = \emptyset$ for every $g \in E\setminus \{ e_G \}$. Let $(F',\phi) \in \mathcal{Z}$ such that $|F'| = \sup \{ |F''| \mid (F'',\psi) \in \mathcal{Z} \}$. We claim that $|F'| = |F|$. On the contrary, assume that $|F'| < |F|$. Then there exists $y \in F\setminus \phi(F')$. Since $G$ is a perfect Hausdorff space, every open non-empty subset of $G$ is infinite. Consequently, there exists some \begin{displaymath}
		x \, \in \, Uy \setminus \left( F' \cup \bigcup\nolimits_{g \in E} gF' \cup g^{-1}F' \right) .
	\end{displaymath} We define $F'' \defeq F' \cup \{ x \}$ and $\psi \colon F'' \to F$ such that $\psi|_{F'} = \phi$ and $\psi(x) = y$. We observe that $(F'',\psi)$ is a member of $\mathcal{Z}$. Since $|F'| < |F''|$, this clearly contradicts our hypothesis. Hence, $|F'| = |F|$ and therefore $\phi$ is a bijection. This finishes the proof. \end{proof}	 

\begin{lemma}\label{lemma:translated.neighborhood} Let $G$ be a topological group. If $F \in \mathcal{F}(G)$, $U \in \mathcal{U}(G)$ and $x \in G$, then \begin{displaymath}
		U[Fx] \, = \, \{ F'x \mid F' \in U[F] \} .
\end{displaymath} \end{lemma}	

\begin{proof} Let $F \in \mathcal{F}(G)$, $U \in \mathcal{U}(G)$ and $x \in G$. If $E \in U[Fx]$, then there exists a bijection $\phi \colon E \to Fx$ such that $y\phi(y)^{-1} \in U_{\Rsh}$ for all $y \in E$, and so $\psi \colon Ex^{-1} \to F, \, y \mapsto \phi(yx)x^{-1}$ is a well-defined bijection satisfying \begin{displaymath}
		y\psi(y)^{-1}\! \, = \, y\!\left( \phi(yx)x^{-1} \right)^{-1}\! \, = \, yx\phi(yx)^{-1}\! \, \in \, U
	\end{displaymath} for all $y \in Ex^{-1}$, whence $Ex^{-1} \in U[F]$ and thus \begin{displaymath}
		E \, = \, Ex^{-1}x \, \in \, \{ F'x \mid F' \in U[F] \} .
	\end{displaymath} Conversely, if $F' \in U[F]$, then there exists a bijective map $\phi \colon F' \to F$ with $y\phi(y)^{-1} \in U$ for all $y \in F'$, so that $\psi \colon F'x \to Fx, \, y \mapsto \phi(yx^{-1})x$ is a well-defined bijection satisfying \begin{displaymath}
		y\psi(y)^{-1}\! \, = \, y\!\left( \phi\!\left(yx^{-1}\right)x \right)^{-1}\! \, = \, yx^{-1}\phi\!\left(yx^{-1}\right)^{-1}\! \, \in \, U
	\end{displaymath} for all $y \in F'x$, which shows that $F'x \in U[Fx]$. This completes the proof. \end{proof}

\begin{lemma}\label{lemma:folner.perturbation} Let $G$ be a topological group, let $E, F_{0} \in \mathcal{F}(G)$, and let $U,V \in \mathcal{U}(G)$ such that $V^{-1}=V$ and $VVV \subseteq U$. Let $W \defeq \bigcap_{g \in E \cup \{ e_G \}} g^{-1}Vg$. Then for all $F \in W[F_{0}]$ and $g \in E$, \begin{displaymath}
		\match_{\Rsh} (F,gF,U) \, \geq \, \match_{\Rsh}(F_{0},gF_{0},W) .
\end{displaymath} \end{lemma}

\begin{proof} Let $F \in W[F_{0}]$ and $g \in E$. Fix any bijection $\phi \colon F \to F_{0}$ such that $(x,\phi(x)) \in W_{\Rsh}$ for all $x \in F$. By Lemma~\ref{lemma:folner.permutations}, there is a permutation $\gamma_{0}$ on $F_{0}$ with $\vert D_{0} \vert = \match_{\Rsh}(F_{0},gF_{0},W)$ for \begin{displaymath}
		D_{0} \, \defeq \, \{ x \in F_{0} \mid (\gamma_{0} (x),gx) \in W_{\Rsh} \} .
	\end{displaymath} Consider $\gamma \defeq \phi^{-1} \circ \gamma_{0} \circ \phi \in \Sym (F)$ and let $D \defeq \{ x \in F \mid (\gamma (x),gx) \in U_{\Rsh} \}$. We claim that $\phi^{-1}(D_{0}) \subseteq D$. To prove this, let $x \in \phi^{-1}(D_{0})$. Then $(\gamma_{0} (\phi(x)),g\phi(x)) \in W_{\Rsh}$. Moreover, by our choice of $\phi$, both \begin{displaymath}
		(\gamma(x),\gamma_{0} (\phi(x))) \, = \, (\gamma(x),\phi(\gamma(x))) \, \in \, W_{\Rsh}
	\end{displaymath} and $(x,\phi(x)) \in W_{\Rsh}$, thus $gx (g\phi(x))^{-1} = gx\phi(x)^{-1}g^{-1} \in gWg^{-1} \subseteq V$, i.e., $(g\phi(x),gx) \in V_{\Rsh}$. It follows that $(\gamma (x),gx) \in U_{\Rsh}$, that is, $x \in D$. This proves our claim. Hence, \begin{displaymath}
		\match_{\Rsh}(F,gF,U) \, \stackrel{\ref{lemma:folner.permutations}}{\geq} \, \vert D \vert \, \geq \, \left\lvert \phi^{-1}(D_{0}) \right\rvert \, = \, \vert D_{0} \vert \, = \, \match_{\Rsh}(F_{0},gF_{0},W) . \qedhere
\end{displaymath} \end{proof}	

\subsection{Proof of Ollagnier's lemma for topological groups} We have prepared everything to prove the announced topological version of Ollagnier's convergence theorem.

\begin{theorem}[Ollagnier lemma]
\label{the:ollagnierslemma}
	Let $G$ be an amenable topological group. Furthermore, let $f\colon \mathcal{F}(G)\to \mathbb{R}_{\geq 0}$ be a continuous and right-invariant function satisfying Shearer's inequality. Then, for any thin F\o lner net $(F_i)_{i\in I}$ in $G$, the following limit exists and satisfies
	\[\lim\nolimits_{i\in I}\frac{f(F_i)}{|F_i|} \, = \, \inf\nolimits_{F\in \mathcal{F}_{+}(G)}\frac{f(F)}{|F|}. \]
\end{theorem}
\begin{remark}
	In \cite{ollagnier1985ergodic} a proof of the theorem in the context of discrete amenable groups is presented. We thus assume in our proof w.l.o.g.~that $G$ is a non-discrete amenable topological group. Furthermore, we abbreviate $\diam M:=\sup_{(x,y)\in M^2}|x-y|$ for $M\subseteq \mathbb{R}$. 
\end{remark}

\begin{proof} To prove the desired convergence, it suffices to show that \begin{equation}\tag{1}\label{claim}
	\forall K \in \mathcal{F}_{+}(G) \colon \quad \limsup\nolimits_{i \in I} \frac{f (F_{i})}{\vert F_{i} \vert} \, \leq \, \frac{f (K)}{\vert K \vert} .
\end{equation} For this purpose, let $K \in \mathcal{F}_+(G)$ and $\epsilon > 0$. By right invariance of $f$, we may and will assume that $K$ contains the neutral element of $G$. Since $K$ is finite and $f$ is continuous, there exists $U \in \mathcal{U}(G)$ such that $U^{-1} = U$, $KK^{-1} \cap UU = \{ e_G \}$ and \begin{equation}\tag{2}\label{approximation0}
	\forall K' \subseteq K \colon \quad \diam f (U[K']) \, \leq \, \frac{\epsilon}{3} .
\end{equation} For each $K' \subseteq K$ and every $x \in G$, right invariance of $f$ readily implies that \begin{equation}\tag{3}\label{approximation0.1}
	\diam f(U[K'x]) \, \stackrel{\eqref{lemma:translated.neighborhood}}{=} \, \diam \{ f(K''x) \mid K'' \in U[K'] \} \, = \, \diam f(U[K']) \, \stackrel{\eqref{approximation0}}{\leq} \, \frac{\epsilon}{3} .
\end{equation} Consider $C \defeq \max\{ 1, \sup \{ f (K') \mid K' \subseteq K \} \}$ and note that, by right invariance of $f$, \begin{equation}\tag{4}\label{supremum}
	\sup \{ f(K'x) \mid K' \subseteq K, \, x \in G \} \, \leq \, C .
\end{equation} Let $\tau \defeq 1- \frac{\epsilon}{3C\vert K \vert}$. Furthermore, let $V \in \mathcal{U}(G)$ such that $V^{-1} = V$ and $V^{3} \subseteq U$. Define \begin{displaymath}
	W \, \defeq \, \bigcap\nolimits_{g \in K \cup K^{-1}} g^{-1}Vg .
\end{displaymath} Since $(F_{i})_{i \in I}$ constitutes a thin F\o lner net in $G$, there exists $i_{0} \in I$ such that \begin{equation}\tag{5}\label{folner}
	\forall i \in I , \, i_{0} \leq i \ \forall g \in K \colon \quad \match_{\Rsh} (F_{i},gF_{i},W) \, \geq \, \tau\vert F_{i} \vert .
\end{equation} 
Let $i \in I$ with $i_{0} \leq i$. Since $f$ is continuous, there exists some $W_{0} \in \mathcal{U}(G)$ such that $W_{0} \subseteq W$ and $\diam f (W_{0}[F_{i}]) \leq \frac{\epsilon}{3}$. By Lemma~\ref{lemma:approximation.by.moving.sets}
and since we assume $G$ to be non-discrete, there exists $F \in W_{0}[F_{i}]$ with $F \cap gF = \emptyset$ for every $g \in K\setminus \{ e_G \}$. In particular, we deduce from $\diam f (W_{0}[F_{i}]) \leq \frac{\epsilon}{3}$ \begin{equation}\tag{6}\label{approximation1}
	\vert f(F_{i}) - f(F) \vert \, \leq \, \frac{\epsilon}{3} .
\end{equation} Furthermore, thanks to Lemma~\ref{lemma:folner.perturbation}, assertion~\eqref{folner}, and our choice of $W_{0} \subseteq W$, it follows that $\match_{\Rsh} (F,gF,U) \geq \match_{\Rsh}(F_{i},gF_{i},W)$ whenever $g \in K$. Hence, for each $g \in K\setminus \{ e_G \}$, there exists an injection $\psi_{g} \colon D_{g} \to gF$ such that $D_{g} \subseteq F$, $\vert D_{g} \vert \geq \tau \vert F \vert$, and \begin{equation}\tag{7}\label{matching}
	\forall x \in D_{g} \colon \quad (x,\psi_{g}(x)) \, \in \, U_{\Rsh} .
\end{equation} For each $g \in K\setminus \{ e_G \}$, consider the bijection $\phi_{g} \colon G \to G$ defined by \begin{displaymath}
	\phi_{g}(x) \, \defeq \, \begin{cases}
		\, \psi_{g}^{-1}(x) & \text{if } x \in \psi_{g}(D_{g}) , \\
		\, \psi_{g}(x) & \text{if } x \in D_{g} , \\
		\, x & \text{otherwise}
	\end{cases} \qquad (x \in G) .
\end{displaymath} Put $D_{e_G} \defeq F$, $\psi_{e_G} \defeq \id_{F}$, and $\phi_{e_G} \defeq \id_{G}$. Let \begin{displaymath}
	\gamma \, \colon \, K \times G \, \longrightarrow \, G, \quad (g,x) \, \longmapsto \, \phi_{g}(gx) .
\end{displaymath} 
We now claim that \begin{equation}\tag{8}\label{injective}
	\forall g_{0},g_{1} \in K \, \forall x \in G \colon \quad\gamma (g_{0},x) = \gamma(g_{1},x) \ \Longrightarrow \ g_{0} = g_{1}.
\end{equation} Indeed, if $x \in G$ and $g_{0},g_{1} \in K$ with $\gamma (g_{0},x) = \gamma(g_{1},x)$, then from $U^{-1} = U$ and \begin{displaymath}
	(g_{0}x,\gamma (g_{0},x)) \, = \, (g_{0}x,\phi_{g_{0}}(g_{0}x)) \, \stackrel{\eqref{matching}}{\in} \, U_{\Rsh} , \qquad (g_{1}x,\gamma (g_{1},x)) \, = \, (g_{1}x,\phi_{g_{1}}(g_{1}x)) \, \stackrel{\eqref{matching}}{\in} \, U_{\Rsh}
\end{displaymath} we infer that $(g_{0}x,g_{1}x) \in (UU)_{\Rsh}$ and therefore $g_{0}g_{1}^{-1} = (g_{0}x)(g_{1}x)^{-1} \in UU$, whence $g_{0} = g_{1}$ as $KK^{-1} \cap UU = \{ e_G \}$. Next, we prove that \begin{equation}\tag{9}\label{disjoint}
	\forall y \in G \colon \quad \vert \{ x \in G \mid \exists g \in K \colon \, \gamma (g,x) = y \} \vert \, = \, \vert K \vert .
\end{equation} To see this, let $y \in G$. For each $g \in K$, the map $G \to G, \, x \mapsto \gamma(g,x)$ is a bijection, whence \begin{displaymath}
	\vert \{ x \in G \mid \gamma (g,x) = y \} \vert \, = \, 1 .
\end{displaymath} Since \begin{displaymath}
	\{ x \in G \mid \exists g \in K \colon \, \gamma (g,x) = y \} \, \stackrel{\eqref{injective}}{=} \, \bigcupdot\nolimits_{g \in K} \{ x \in G \mid  \gamma (g,x) = y \} ,
\end{displaymath} it follows that \begin{displaymath}
	\vert \{ x \in G \mid \exists g \in K \colon \, \gamma (g,x) = y \} \vert \, = \, \sum\nolimits_{g \in K} \vert \{ x \in G \mid \gamma (g,x) = y \} \vert \, = \, \vert K \vert .
\end{displaymath} This proves~\eqref{disjoint}. Now, we consider the finite set \begin{displaymath}
	H \, \defeq \, \{ x \in G \mid \gamma (K \times \{x \}) \cap F \ne \emptyset \} \, = \, \bigcup\nolimits_{g \in K} g^{-1}\phi_{g}^{-1}(F) 
\end{displaymath} and note that \begin{equation}\tag{10}\label{cardinality}
	\vert H \vert \, \leq \, \sum\nolimits_{g \in K} \left\lvert g^{-1}\phi_{g}^{-1}(F) \right\rvert \, = \, \vert K \vert \vert F \vert .
\end{equation} We observe that, by~\eqref{disjoint}, \begin{displaymath}
	\forall y \in F \colon \quad \vert \{ x \in H \mid y \in \gamma(K \times \{ x \}) \cap F \} \vert \, = \, \vert K \vert ,
\end{displaymath} 
which entails that
$(\gamma(K\times \{x\})\cap F)_{x\in H}$ is a $|K|$-covering of $F$. 
Since $f$ satisfies Shearer's inequality we observe  \begin{equation}\tag{11}\label{shearer}
	f (F) \, \leq \, \frac{1}{\vert K \vert} \sum\nolimits_{x \in H} f (\gamma (K \times \{x \}) \cap F) .
\end{equation} Moreover, for each $x \in G$, we have $\vert \gamma (K \times \{ x \}) \vert = \vert K \vert$ due to~\eqref{injective}, and then \begin{equation}\tag{12}\label{neighborhood}
	\gamma (K \times \{x \}) \, \in \, U[Kx]
\end{equation} by~\eqref{matching} and $U=U^{-1}$, wherefore \begin{equation}\tag{13}\label{approximation2}
	\vert f(\gamma (K \times \{x \})) - f(Kx) \vert \, \stackrel{\eqref{approximation0.1}}{\leq} \, \frac{\epsilon}{3} .
\end{equation} Also, if $x \in G$, then~\eqref{neighborhood} implies that $\gamma (K \times \{ x\}) \cap F \in U[K'x]$ for some $K' \subseteq K$, so that \begin{displaymath}
	\vert f(\gamma (K \times \{x \}) \cap F) - f(K'x) \vert \, \stackrel{\eqref{approximation0.1}}{\leq} \, \frac{\epsilon}{3}
\end{displaymath} and thus \begin{equation}\tag{14}\label{approximation3}
	f(\gamma (K \times \{x \}) \cap F) \, \stackrel{\eqref{supremum}}{\leq} \, C + \frac{\epsilon}{3} .
\end{equation} Furthermore, since $e_G \in K$ and $\phi_{e_G} = \id_{G}$, \begin{displaymath}
	H_{0} \, \defeq \, \{ x \in G \mid \gamma (K \times \{x \}) \subseteq F \} \, = \, \bigcap\nolimits_{g \in K} g^{-1}\phi_{g}^{-1}(F) \, = \, \bigcap\nolimits_{g \in K} \{ x \in F \mid \phi_{g}(gx) \in F \} .
\end{displaymath} Now, if $g \in K$ and $x \in g^{-1}\psi_{g}(D_{g})$, then $\phi_{g}(gx) = \psi^{-1}_{g}(gx) \in F$. This entails that \begin{displaymath}
	\bigcap\nolimits_{g \in K} g^{-1}\psi_{g}(D_{g}) \, \subseteq \, \bigcap\nolimits_{g \in K} \{ x \in F \mid \phi_{g}(gx) \in F \} \, = \, H_{0} 
\end{displaymath} and, in turn, \begin{align*}
	\vert H_{0} \vert \, &\geq \, \left\lvert \bigcap\nolimits_{g \in K} g^{-1}\psi_{g}(D_{g}) \right\rvert \, = \, \vert F \vert - \left\lvert \bigcup\nolimits_{g \in K} F \setminus g^{-1}\psi_{g}(D_{g}) \right\rvert \\
		& \geq \, \vert F \vert - \sum\nolimits_{g \in K} \left\lvert F \setminus g^{-1}\psi_{g}(D_{g}) \right\rvert \, = \, \vert F \vert - \sum\nolimits_{g \in K} \vert F \setminus D_{g} \vert \\
		& \geq \, \vert F \vert - \vert K \vert (1-\tau)\vert F \vert \, = \, \left(1-\frac{\epsilon}{3C}\right) \vert F \vert .
\end{align*} Hence, considering the set \begin{displaymath}
	H_{1} \, \defeq \, H\setminus H_{0} \, = \, \bigcup\nolimits_{g \in K} \left(g^{-1}\phi_{g}^{-1}(F)\right)\!\setminus \, \bigcap\nolimits_{{g}' \in K} \left(g'^{-1}\phi_{g'}^{-1}(F)\right)\! ,
\end{displaymath} we conclude that \begin{equation}\tag{15}\label{estimate}
	\vert H_{1} \vert \, \leq \, \sum\nolimits_{g \in K} \left\lvert \left(g^{-1}\phi_{g}^{-1}(F)\right)\!\setminus H_{0} \right\rvert \, \leq \, \vert K \vert \left( \vert F\vert - \vert H_{0} \vert \right) \, \leq \, \frac{\epsilon \vert F \vert \vert K \vert}{3C} .
\end{equation} Consequently, it follows that \begin{align*}
	f (F_{i}) \, &\stackrel{\eqref{approximation1}}{\leq} \, f (F) + \frac{\epsilon}{3} \, \stackrel{\eqref{shearer}}{\leq} \, \frac{1}{\vert K \vert} \sum\nolimits_{x \in H} f (\gamma (K \times \{ x \}) \cap F) + \frac{\epsilon}{3} \\
		&= \, \frac{1}{\vert K \vert} \sum\nolimits_{x \in H_{0}} f (\gamma (K \times \{ x \}) \cap F) + \frac{1}{\vert K \vert} \sum\nolimits_{x \in H_{1}} f(\gamma (K \times \{ x \}) \cap F) + \frac{\epsilon}{3} \\
		&= \, \frac{1}{\vert K \vert} \sum\nolimits_{x \in H_{0}} f (\gamma (K \times \{ x \})) + \frac{1}{\vert K \vert} \sum\nolimits_{x \in H_{1}} f(\gamma (K \times \{ x \}) \cap F) + \frac{\epsilon}{3} \\
		&\stackrel{\eqref{approximation2}+\eqref{approximation3}}{\leq} \, \frac{1}{\vert K \vert} \sum\nolimits_{x \in H_{0}} f (Kx) + \frac{\epsilon \vert H_{0} \vert}{3\vert K\vert} + \frac{\left(C+\frac{\epsilon}{3}\right) \vert H_{1} \vert}{\vert K \vert} + \frac{\epsilon}{3} \\
		&= \, \frac{1}{\vert K \vert} \sum\nolimits_{x \in H_{0}} f (K) + \frac{\epsilon \vert H \vert}{3\vert K\vert} + \frac{C \vert H_{1} \vert}{\vert K \vert} + \frac{\epsilon}{3} \\
		&\stackrel{\eqref{cardinality}+\eqref{estimate}}{\leq} \, \frac{f (K)\vert H_{0} \vert}{\vert K \vert} + \frac{\epsilon \vert F \vert}{3} + \frac{\epsilon \vert F \vert}{3} + \frac{\epsilon}{3} \\
		&\leq \, \frac{f (K)\vert H_{0} \vert}{\vert K \vert} + \epsilon \vert F \vert \, \leq \, \frac{f (K)\vert F \vert}{\vert K \vert} + \epsilon \vert F \vert ,
\end{align*} and therefore \begin{displaymath}
	\frac{f (F_{i})}{\vert F_{i} \vert} \, = \, \frac{f (F_{i})}{\vert F \vert} \, \leq \, \frac{f (K)}{\vert K \vert} + \epsilon .
\end{displaymath} This proves~\eqref{claim} and hence the theorem. \end{proof}

\section{The Ornstein-Weiss lemma}
\label{sec:OWlemma}

	As described in the introduction we will next present a proof of the Ornstein-Weiss lemma for non-discrete locally compact groups using quasi-tiling techniques, which is based on  \cite{ornstein1987entropy, gromov1999topological, lindenstrauss2000mean, krieger2007lemme, ceccherini2014analogue}. Throughout this section, let $G$ be an amenable unimodular locally compact group and recall that $\theta$ denotes a Haar measure on $G$.
	
\subsection{On $\epsilon$-disjointness and invariance}
		
	Let $\epsilon>0$. A finite family $(A_i)_{i \in I}$ of  compact subsets of $G$ is called \emph{$\epsilon$-disjoint} if there exists a family of compact subsets $B_{i} \subseteq A_{i}$ $(i \in I)$ such that $(B_i)_{i \in I}$ are pairwise disjoint and $\theta(B_i)>(1-\epsilon)\theta(A_i)$ for each $i \in I$. For every $\epsilon$-disjoint finite family $(A_i)_{i \in I}$ of compact subsets of $G$, it is straightforward to show that
	\[(1-\epsilon)\sum\nolimits_{i \in I} \theta(A_i)\,\leq \, \theta\!\left(\bigcup\nolimits_{i \in I} A_i\right).\]
	The lemmas in this subsection are standard and can for example be found in \cite{ornstein1987entropy, krieger2007lemme}. We include the short proofs for the convenience of the reader. 
	
	\begin{lemma}\label{lem:constructionepsilondisjointfamilies}
	Let $(A_i)_{i\in F}$ be a finite and $\epsilon$-disjoint family of compact subsets of $G$. Then, for every compact $A\subseteq G$ with $\theta(A\cap (\bigcup_{i\in F} A_i))<\epsilon\theta(A)$, the extended finite family $(A_i)_{i \in I} \sqcup (A)$ is $\epsilon$-disjoint. 
\end{lemma}
\begin{proof}
	As $(A_i)_{i\in F}$ is $\epsilon$-disjoint there exist disjoint compact subsets $B_i\subseteq A_i$ $(i \in F)$ such that $\theta(B_i)> (1-\epsilon)\theta(A_i)$ for each $i \in F$. 
	By our assumptions, 
	\[\theta\!\left(A\setminus \left(\bigcup\nolimits_{i\in F} A_i\right)\right) \! \, \geq \, \theta(A)-\theta\!\left(\bigcup\nolimits_{i\in F}A_i\right) \! \, > \, (1-\epsilon)\theta(A) . \] Thus, there is $\rho>0$ such that
	$(1-\rho)\theta(A\setminus \bigcup_{i\in F} A_i)>(1-\epsilon)\theta(A)$.
	As $\theta$ is regular, there exists a compact subset $B\subseteq A\setminus \bigcup_{i\in F}A_i$ such that $\theta(B)\geq (1-\rho)\theta(A\setminus \bigcup_{i\in F}A_i)\geq (1-\epsilon)\theta(A)$. For each $i \in I$, the set $B$ is disjoint from $A_i$, thus disjoint from $B_{i}$. Hence, $(A_i)_{i \in I} \sqcup (A)$ is $\epsilon$-disjoint.
\end{proof}

\begin{lemma}\label{lem:alphaandepsilondisjointfamilis}
	Let $K\in \mathcal{K}(G)$ and $\epsilon\in(0,1)$. If $(A_i)_{i\in F}$ is an $\epsilon$-disjoint finite family of non-empty, compact subsets of $G$, then 
	\[\alpha\!\left(\bigcup\nolimits_{i\in F} A_i,K\right)\! \, \leq \, \frac{\max_{i\in F} \alpha(A_i,K)}{1-\epsilon}.\]
\end{lemma}
\begin{proof}
	Let us abbreviate $M\defeq\max_{i} \alpha(A_i,K)=\max_i \theta(\partial_K A_i)/\theta(A_i)$. 
	A straightforward argument shows that $\partial_K (\bigcup_{i\in F}A_i)\subseteq \bigcup_{i\in F}\partial_K A_i$, and hence \begin{align*}
		\theta \! \left(\partial_K \! \left(\bigcup\nolimits_{i\in F}A_i\right)\right)
	\! \, &\leq \, \theta\!\left( \bigcup\nolimits_{i\in F}\partial_K A_i\right) \\
	& \leq \, \sum\nolimits_{i\in F}\theta(\partial_K A_i)
	\, = \, \sum\nolimits_{i\in F}\theta(A_i)\frac{\theta(\partial_K A_i)}{\theta(A_i)}
	\, \leq \, M \sum\nolimits_{i\in F}\theta(A_i).
	\end{align*}
	Since the considered family is $\epsilon$-disjoint, we obtain that $(1-\epsilon) \sum_{i\in F}\theta(A_i)\leq \theta\left(\bigcup_{i\in F} A_i\right)$ and thus conclude that
	\[\alpha\!\left(\bigcup\nolimits_{i\in F} A_i,K\right)\! \, = \, \frac{\theta\!\left(\partial_K \bigcup\nolimits_{i\in F} A_i \right)}{\theta\!\left(\bigcup\nolimits_{i\in F} A_i \right)}
	\, \leq \, \frac{M \sum_{i\in F}\theta(A_i)}{(1-\epsilon)\sum_{i\in F}\theta(A_i)}
	\, = \, \frac{M}{1-\epsilon}. \qedhere\]
\end{proof}

\begin{lemma}\label{lem:alphaandsetminus}
	Let $K\in \mathcal{K}(G)$ and $\epsilon>0$. 
	Any precompact and measurable subsets $R,B\subseteq G$ with $0<\theta(B)\leq \theta(R)$ and $\theta(R\setminus B)\geq \epsilon\theta(R)$ satisfy
	\[\alpha(R\setminus B,K)\leq \frac{\alpha(R,K)+\alpha(B,K)}{\epsilon}.\]
\end{lemma}
\begin{proof}
	A straightforward argument shows that $\partial_K\left(R\setminus B\right)\subseteq \partial_K R\cup \partial_K B$. Thus, if $\epsilon>0$ and $\theta(R\setminus B)\geq \epsilon\theta(R)$, then we conclude that
	\[\frac{\theta\left(\partial_K \left(R\setminus B\right)\right)}{\theta(R\setminus B)}
	\, \leq \, \frac{1}{\epsilon}\frac{\theta(\partial_K R)+\theta(\partial_K B)}{\theta(R)}
	\, \leq \, \frac{1}{\epsilon}\left(\frac{\theta(\partial_K R)}{\theta(R)}+\frac{\theta(\partial_K B)}{\theta(B)}\right). \qedhere \]
\end{proof}

	\subsection{On $(\epsilon,A)$-fillings}

	Let $A,R\subseteq G$ be measurable subset with positive Haar measure and let $\epsilon>0$. Assume $A$ to be compact. 
	We call $C\subseteq G$ an \emph{$(\epsilon,A)$-filling of $R$} if $AC\subseteq R$ and the family $(Ag)_{g \in C}$ is $\epsilon$-disjoint. 
	If $G$ is discrete, then every precompact subset $R$ of $G$ is finite, and the cardinality of every $(\epsilon,A)$-filling $C$ of $R$ is bounded by $|R|$. 
	The next lemma shows that we can bound the cardinality of such $C$ even without the discreteness assumption.

\begin{lemma}\label{lem:cardinalitybound}
	Let $\epsilon>0$, $A\subseteq G$ be a compact subset of positive Haar measure and $R\subseteq G$ be a precompact, measurable subset. Then every $(\epsilon,A)$-filling $C$ of $R$ satisfies 
	\[|C|\leq \frac{\theta(R)}{(1-\epsilon)\theta(A)}.\] 
	In particular there are $(\epsilon,A)$-fillings of $R$ of finite maximal cardinality. 
\end{lemma}
\begin{proof}
	As $(Ag)_{g \in C}$ is an $\epsilon$-disjoint family, we obtain 
	\[\theta(R) \, \geq \, \theta(AC) \, = \, \theta\!\left(\bigcup\nolimits_{g\in C}Ag\right) \, \geq \, (1-\epsilon) \sum\nolimits_{g\in C} \theta(Ag) \, = \, (1-\epsilon)|C|\theta(A). \qedhere \]
\end{proof}

	The idea of the proof of the next lemma is sketched in \cite{gromov1999topological} and given in detail for discrete groups in \cite{krieger2007lemme, krieger2010ornstein, ceccherini2014analogue}. We include a full prove for the convenience of the reader. 
	
\begin{lemma} \label{lem:fillinglemma}
Let $A\subseteq G$ be a compact subset, $R\subseteq G$ be a precompact, measurable subset and assume both sets to have positive Haar measure. Let furthermore $\epsilon\in (0,1)$.  Then, for every $(\epsilon,A)$-filling $C$ of $R$ of finite maximal cardinality, we have
	\[\theta(AC)\,\geq \, \epsilon\left(1-\alpha(R,A^{-1})\right) \theta(R).\]
\end{lemma}
\begin{proof}
	Since $\theta(A)>0$, there exists $a\in G$ such that $Aa$ contains the identity $e_G$. Then $\theta(\partial_{(Aa)^{-1}}R)=\theta(a^{-1}\partial_{A^{-1}} R)=\theta(\partial_{A^{-1}} R)$, and upon translating $C$ we may and will assume without lost of generality that $A$ contains $e_G$. 
	For $g\in R\setminus \partial_{A^{-1}}R$, we have $g\in R\subseteq A^{-1}\overline{R}$ and thus $g\notin A^{-1}\overline{R^c}$. 
	In particular, $Ag\cap \overline{R^c}$ is empty, and we deduce that $Ag\subseteq R$. 
	If $g\notin C$, then necessarily 
	$\theta(Ag\cap AC)\geq \epsilon\theta(Ag)$, as otherwise, by Lemma \ref{lem:constructionepsilondisjointfamilies}, $(Ag')_{g'\in C\cup\{g\}}$ would be an $\epsilon$-disjoint family, a contradiction to the maximal cardinality of $C$. 
	For $g\in C$ we furthermore obtain $Ag\subseteq AC$ and thus $\theta(Ag\cap AC)=\theta(Ag)\geq \epsilon\theta(A)$ from $\epsilon\in(0,1)$. 
	This shows that, for every $g\in R\setminus \partial_{A^{-1}}R$, 
	\[
	\theta(Ag\cap AC) \, \geq \, \epsilon\theta(A). 
	\]
	Let now $\chi_M$ denote the characteristic function of a subset $M\subseteq G$. For any $g'\in G$, we compute that
	\[\theta(A)
	\, = \, \theta(A^{-1})
	\, = \, \int_G  \chi_{A}(g^{-1})d\theta(g)
	\, = \, \int_G  \chi_{A}(g'g^{-1})d\theta(g).\]
Thus, Tonelli's theorem implies
	\begin{align*}
	\theta(A)\theta(AC)
	&=\int_G \theta(A)\chi_{AC}(g') d\theta(g')
	=\int_G  \int_G \chi_{A}(g'g^{-1})d\theta(g) \chi_{AC}(g') d\theta(g')\\
	&=\int_G  \int_G \chi_{A}(g'g^{-1}) \chi_{AC}(g') d\theta(g') d\theta(g)
	=\int_G \int_G \chi_{Ag\cap AC}(g') d\theta(g')d\theta(g)\\
	&=\int_G\theta(Ag\cap AC)d\theta(g)
	\geq \int_{R\setminus \partial_{A^{-1}}R} \epsilon\theta(A)d\theta(g)=\epsilon\theta(A)\theta(R\setminus \partial_{A^{-1}}R).
	\end{align*}
	\[\]
	Consequently, we arrive at
	\[\theta(AC) \, \geq \, \epsilon \theta(R\setminus \partial_{A^{-1}}R) \, \geq \, \epsilon(\theta(R)-\theta(\partial_{A^{-1}}R)) \, = \, \epsilon\left(1-\frac{\theta(\partial_{A^{-1}} R)}{\theta(R)}\right) \theta(R). \qedhere \]
\end{proof}

\subsection{On $\epsilon$-quasi-tiling}
	Let $A\in \mathcal{K}(G)$ and let $\epsilon>0$. 
	A finite family $(A_i)_{i\in F}$ of compact subsets of $G$ is an \emph{$\epsilon$-quasi-tiling} of $A$ if there exists a family of finite sets $(C_i)_{i\in F}$ such that 
	\begin{itemize}
		\item[(a)] for every $i \in F$, the family $(A_ig)_{g\in C_i}$ is $\epsilon$-disjoint;
		\item[(b)] $(A_iC_i)_{i \in F}$ is a disjoint family; and
		\item[(c)] $\theta(A\cap \bigcup_{i\in F} A_iC_i)\geq (1-\epsilon) \theta(A)$.		
	\end{itemize}	
	A family $(C_i)_{i\in F}$ satisfying these conditions is referred to as a family of \emph{$\epsilon$-quasi-tiling centres} of $(A_i)_{i\in F}$ (with respect to $A$). 
		In \cite{ornstein1983shannon, ornstein1987entropy, lindenstrauss2000mean} it is shown that, for any van Hove sequence $(A_n)_{n\in \mathbb{N}}$ and any $\epsilon>0$, one can find a finite subset $F\subseteq \mathbb{N}$ such that, for every sufficiently invariant $A$, the family $(A_i)_{i\in F}$ $\epsilon$-quasi-tiles $A$. We will follow the ideas that lead to this concept and show that, with some modifications of these methods, one can also construct $\epsilon$-quasi-tiling centres $C_i$ such that $R \defeq A\setminus \bigcup_{i\in F} A_iC_i$ is 'relatively invariant', i.e.\ that allows to control $\theta(\partial_K R)/\theta(A)$ for some given compact subset $K\subseteq G$.

\begin{proposition}\label{pro:quasitiling}
	For any van Hove net $(A_i)_{i\in I}$, any cofinal subset $J\subseteq I$, any $\epsilon\in (0,1/2)$ and any non-empty and compact subset $K\subseteq G$, there exist a finite subset $F\subseteq J$, $\delta>0$ and a compact subset $D\subseteq G$ with the following property. 
	For any $(\delta,D)$-invariant and compact subset $A\subseteq G$ the finite family $(A_i)_{i\in F}$ is an $\epsilon$-quasi-tiling of $A$ and the $\epsilon$-quasi-tiling centres $(C_i)_{i\in F}$ can be chosen such that $\bigcup_{i\in F} A_iC_i\subseteq A$ and such that $R \defeq A\setminus \bigcup_{i\in F} A_i C_i$ satisfies $\theta(R)+\theta(\partial_K R)\leq \epsilon\theta(A).$
\end{proposition}

\begin{proof}
	As $\epsilon\in (0,1/2)$, there exists $N\in \mathbb{N}$ such that $\left(1-{\epsilon}/2\right)^{N}\leq \epsilon/2$.
	As $(A_i)_{i\in J}$ is a van Hove net, we can choose inductively $i_n\in J$ for $n={{N}},\ldots,1$ such that $A_{i_n}$ has positive Haar measure, such that the $i_n$ are pairwise distinct and such that
	\begin{align*}
	\alpha\left(A_{i_n},K\cup \left(\bigcup\nolimits_{m=n+1}^{{{N}}}A_{i_m}^{-1}\right)\right)
	\leq \epsilon^{2(N-n)+4}. 
	\end{align*}
	We abbreviate $K_n\defeq K\cup \left(\bigcup_{m=n+1}^{{{N}}}A_{i_m}^{-1}\right)$
	and observe that $A_{i_n}$ is $(\epsilon^{2(N-n)+4}, K_n)$-invariant for $n=0,\ldots,{{N}}$. Furthermore we have $K_0\supseteq K_1\supseteq \ldots \supseteq K_N=K$ and $A_{i_n}^{-1}\subseteq K_{n-1}$ for $n=1,\ldots,N$.

We set $F\defeq \{i_1,\ldots,i_N\}$, $\delta \defeq \epsilon^{2{N}+1}$ and $D\defeq K_0$ and consider a compact and $(\delta,D)$-invariant subset $A$ of $G$. 
	Set $R_0\defeq A$. 
	Using Lemma \ref{lem:cardinalitybound} we now choose inductively for $n=1,\ldots,M$ finite $(\epsilon,A_{i_n})$-fillings $C_n$ of $R_{n-1}$ of maximal cardinality, where we abbreviate 
	$R_n\defeq R_{n-1}\setminus A_{i_n}C_n$ and $M\leq N$ is the smallest integer where our choices lead to  
	$\theta(R_M)=\theta(R_{M-1}\setminus A_{i_M}C_M)\leq\epsilon\theta(R_{M-1})$ and $N$ if we never encounter this situation. 
	Thus, in particular, for $n=1,\ldots,M-1$ there holds $\theta(R_n)>\epsilon\theta(R_{n-1})$. For $n\in \{M+1,\ldots,N\}$ we set $C_n\defeq \emptyset$. 
	Note that $R_n$ is precompact and measurable for $n=0,\ldots,M$ and that there holds
	\[R \, = \, R_M \, \subseteq \, R_{M-1} \, \subseteq \, R_0 \, = \, A.\]
	We will now show that defining $C_{i_n}\defeq C_n$ we get that $(C_i)_{i\in F}=(C_n)_{n=1}^N$ is a family of $\epsilon$-quasi-tiling centres that fulfils the required properties. 
	
	We first show inductively for $n=0,\ldots,M-1$ that $R_n$ is $(\epsilon^{2(N-n)+1},K_n)$-invariant, i.e.\
	\begin{align}\label{form:Rninvariant}
		\alpha(R_n,K_n) \, \leq \, \epsilon^{2(N-n)+1}.
	\end{align}
	This is clearly satisfied for $n=0$, as $R_0=A$, $K_0=D$ and $\epsilon^{2(N-0)+1}=\delta$. 
	To proceed inductively we assume $R_n$ to be $(\epsilon^{2(N-n)+1},K_n)$-invariant for some $n<M-1$ and as $K_{n+1}\subseteq K_n$ we obtain 
	\[\alpha(R_n,K_{n+1})\leq \epsilon^{2(N-n)+1}.\]
	Now recall that $A_{i_{n+1}}$ is $(\epsilon^{2(N-n)+4},K_{n+1})$-invariant. 
	As $C_{n+1}$ is an $(\epsilon,A_{i_{n+1}})$-filling we obtain that $(A_{i_{n+1}}g)_{g \in C_{n+1}}$ is $\epsilon$-disjoint and apply Lemma \ref{lem:alphaandepsilondisjointfamilis} to see 
	\begin{align*}
\alpha(A_{i_{n+1}}C_{n+1},K_{n+1}) \,
	&\leq \, \frac{1}{1-\epsilon}\max_{g\in C_{n+1}}\alpha(A_{i_{n+1}} g,K_{n+1})\\
	&= \, \frac{\alpha(A_{i_{n+1}},K_{n+1})}{1-\epsilon}\\
	&\leq \, \frac{\epsilon^{2(N-n)+4}}{1-\epsilon}
	\, \leq \, \epsilon^{2(N-n)+1}.
	\end{align*}
	For this we have used that $\epsilon<1/2$ yields that $\epsilon/(1-\epsilon)<1$.
	 As we assume $n<M-1$ we obtain $\theta(R_{n}\setminus A_{i_{n+1}}C_{n+1})=\theta(R_{n+1})> \epsilon\theta(R_{n})$. Thus, Lemma~\ref{lem:alphaandsetminus} yields that
	\begin{align*}
	\alpha(R_{n+1},K_{n+1}) \,
	&= \, \alpha\left(R_{n}\setminus A_{i_{n+1}}C_{n+1},K_{n+1}\right)\\
	&\leq \, \frac{\alpha(R_{n},K_{n+1})+\alpha(A_{i_{n+1}}C_{n+1},K_{n+1})}{\epsilon}\\
	&\leq \, 2\epsilon^{2(N-n)}
	\, \leq \, \epsilon^{2(N-n)-1}
	\, = \, \epsilon^{2(N-(n+1))+1} ,
	\end{align*}
	and we have completed the induction to show (\ref{form:Rninvariant}). 
	
	We next show that
	\begin{align}\label{form:bla}
		\theta(R) \, \leq \, \frac{\epsilon}{2}\theta(A). 
	\end{align}
	This statement is satisfied whenever $M<N$, as then we obtain that
	\[\theta(R)=\theta(R_{M})\leq \epsilon \theta(R_{M-1})\leq \epsilon\theta (A).\]
	In order to show (\ref{form:bla}) we thus assume without lost of generality that $M=N$ and that $\theta(R)>0$. Then $\theta(R_N)=\theta(R)>0$ and, for $n\in \{1,\ldots,N-1\}$, we obtain
	\[ \theta(R_n)
	\, \geq \, \epsilon\theta(R_{n-1})
	\, \geq \, \epsilon^n \theta(R_0)
	\, = \, \epsilon^n\theta(A)
	\, > \, 0.\] 
	For $n\leq M=N$ we have chosen $C_n$ to be an $(\epsilon,A_{i_n})$-filling of $R_{n-1}$ of maximal cardinality. Thus we obtain from Lemma \ref{lem:fillinglemma}, (\ref{form:Rninvariant}) and $A_{i_n}^{-1}\subseteq K_{n-1}$ that 
\[\frac{\theta(A_{i_n} C_n)}{\theta(R_{n-1})}
	\, \geq \, \epsilon\left(1-\alpha(R_{n-1},A_{i_n}^{-1})\right)
	\, \geq \, \epsilon\left(1-\alpha(R_{n-1},K_{n-1})\right) \, \geq \, \epsilon\left(1-\epsilon^{2(N-(n-1))+1}\right) \, \geq \, \frac{\epsilon}{2},\]
Thus,
	\[\theta(R_n) \, = \, \theta(R_{n-1})-\theta(A_{i_n}C_n) \, < \, \left(1-\frac{\epsilon}{2}\right)\theta(R_{n-1})\]
	and we obtain from our choice of $N$ that
	\[\theta(R) \, = \, \theta(R_{N}) \, < \, \left(1-\frac{\epsilon}{2}\right)^{N}\theta(R_0) \, \leq \, \frac{\epsilon}{2}\theta(A).\]
	This shows the claimed statement (\ref{form:bla}). 
	
	As $C_n$ is an $(\epsilon,A_{i_n})$-filling of $R_{n-1}$ we obtain from the construction $R_n\defeq R_{n-1}\setminus A_{i_n}C_n$ that $(A_{i_n}g)_{g\in C_n}$ is $\epsilon$-disjoint for all $n\leq M$ and that $(A_{i}C_i)_{i\in F} = (A_{i_n}C_n)_{n=1}^{M} \sqcup (\emptyset)_{n=M+1}^{N}$ is a disjoint family. In particular, we observe that $(A_ig)_{g\in C_i}$ is $\epsilon$-disjoint for all $i\in F$. Furthermore, one obtains 
	\[\bigcup_{i\in F} A_iC_i=\bigcup_{n=1}^M A_{i_n} C_n\subseteq \bigcup_{n=0}^M R_n= R_0=A.\]
	 Thus (\ref{form:bla}) allows to compute 
	\[\theta\left(A\cap \bigcup\nolimits_{i\in F}A_iC_i\right)=\theta(A)-\theta(R)\geq (1-\epsilon)\theta(A).\]
	This shows that $(A_i)_{i\in F}$ is an $\epsilon$-quasi-tiling of $A$ and that $\bigcup_{i\in F} A_iC_i\subseteq A$. 
	
	By (\ref{form:bla}) it remains to show that $\theta(\partial_K R)\leq (\epsilon/2) \theta(A).$
	Recall that $A_{i_M}$ is $(\epsilon^{2(N-M)+4},K_{M})$-invariant. As $K\subseteq K_M$ we obtain 
	\[\alpha(A_{i_M},K)\leq \epsilon^4.\]
	Since $C_{M}$ is an $(\epsilon,A_{i_M})$-filling we observe that $(A_{i_M}g)_{g\in C_{M}}$ is $\epsilon$-disjoint and apply Lemma \ref{lem:alphaandepsilondisjointfamilis} to see 
	\begin{align*}
\alpha(A_{i_M}C_{M},K)
	&\leq \frac{1}{1-\epsilon}\max_{g\in C_{M}}\alpha(A_{i_M} g,K)\\
	&=\frac{\alpha(A_{i_M},K)}{1-\epsilon}\\
	&\leq \frac{\epsilon^{4}}{1-\epsilon}\\
	&\leq \epsilon^3.
	\end{align*}
	Furthermore a straightforward argument shows that $\partial_K R_M=\partial_K (R_{M-1}\setminus A_{i_M}C_M)\subseteq \partial_K R_{M-1} \cup \partial_K (A_{i_M}C_M)$. As $K\subseteq D$ we obtain that $A$ is $(K,\delta)$-invariant and as $\delta\in (0,1)$ we have $\theta(K)\leq \theta(A)$.
	Thus (\ref{form:Rninvariant}) and $K\subseteq K_M \subseteq K_{M-1}$ allow to compute 
	\begin{align*}
	\frac{\theta(\partial_K R)}{\theta(A)}
	&\leq \frac{\theta(\partial_K R_{M-1})}{\theta(A)}+ \frac{\theta(\partial_K A_{i_M}C_M)}{\theta(A)}\\
	&\leq \alpha(R_{M-1},K)+\alpha(A_{i_M}C_M,K)\\
	&\leq \alpha(R_{M-1},K_{M-1})+\epsilon^3\\
	&\leq \epsilon^{3} +\epsilon^3
	\leq \frac{\epsilon}{2}. \qedhere
	\end{align*}
\end{proof}

\subsection{A proof of the Ornstein-Weiss lemma for non-discrete groups}

	If $G$ is a discrete group, then subadditivity and right invariance can be used to show that $f(F)/|F|\leq f(\{e_G\})$ for every finite, non-empty subset $F\subseteq G$. This is not any longer the case in the non-discrete setting as presented in \cite[Remark 3.1 and Remark 4.6]{hauser2020Anote}. We next present a result that can serve as a replacement of the mentioned boundedness. 
	
\begin{lemma}\label{lem:controlf}
	Let $f\colon \mathcal{K}(G)\to[0,\infty)$ be a monotone, right invariant and subadditive mapping, let $K$ be a compact neighbourhood of $e_G$. Then there exists a constant $c_K>0$ such that, for every non-empty, precompact, measurable subsets $R\subseteq G$, \begin{displaymath}
			f\left(\overline{R}\right) \, \leq \, c_K \theta\left(K\overline{R}\right).
	\end{displaymath}
\end{lemma}
\begin{proof}
	Note first that the subadditivity and the monotonicity of $f$ imply that $f$ only takes values in $[0,\infty)$. 
	Let $V$ be a compact and symmetric neighbourhood of $e_G$ that satisfies $V\subseteq K$ and set $c_K\defeq f(VV)/\theta(V)$. 
	For a non-empty and precompact subset $R\subseteq G$ we let $F\subseteq \overline{R}$ be a \emph{$V$-separated} subset, i.e.\ such that $(Vg)_{g\in F}$ is a disjoint family. Then $VF=\bigcup_{g\in F}Vg$ is a disjoint union and thus 
	\[\theta\left(V\overline{R}\right) \, \geq \, \theta(VF) \, = \, \sum_{g\in F}\theta(Vg) \, = \, |F|\theta(V).\]
	In particular we obtain the cardinality of $F$ to be bounded by $\theta(V\overline{R})/\theta(V)< \infty$. We thus assume without lost of generality that $F$ is a $V$-separated subset of $\overline{R}$ of maximal cardinality. Then for any $g\in \overline{R}$ there is $g'\in F$ such that $Vg$ and $Vg'$ intersect. Thus, $g\in VVg'\subseteq VVF$ and we observe that $\overline{R}\subseteq VVF$. 
	We compute that \begin{align*}
		f\left(\overline{R}\right) \, \leq \, f(VVF) \, &\leq \, \sum\nolimits_{g\in F}f(VVg) \, = \, |F|f(VV) \\
		& \leq \, \frac{\theta(V\overline{R})}{\theta(V)}f(VV) \, = \, c_K\theta(V\overline{R}) \, \leq \, c_K\theta(K\overline{R}).\qedhere
	\end{align*}
\end{proof}

	With the slight improvement of the Ornstein-Weiss quasi-tiling machinery at hand, we are now ready to prove the Ornstein-Weiss lemma for arbitrary amenable unimodular locally compact groups. 

\begin{theorem}[Ornstein-Weiss lemma]
\label{the:Ornstein-Weisslemma}
	Let $G$ be an amenable unimodular locally compact group and let
	$f\colon \mathcal{K}(G)\to \mathbb{R}$ be a monotone, subadditive and right-invariant function. Then for any van Hove net $(A_i)_{i\in I}$ in $G$ the limit 
	\[\lim_{i\in I}\frac{f(A_i)}{\theta(A_i)}\]
	exists, is finite and independent of the choice of a van Hove net. 
\end{theorem}

\begin{proof}
	Let $\epsilon\in (0,1/2)$ and choose an arbitrary compact neighbourhood $K$ of $e_G$. Then, by Lemma \ref{lem:controlf}, there exists a constant $c$ such that $f(\overline{A})\leq c \theta(K\overline{A})$ for every non-empty, precompact, measurable subset $A\subseteq G$. Consider $\lambda\defeq \liminf_{i\in I}{f(A_i)}/{\theta(A_i)}$. As $(KA_i)_{i\in I}$ is a van Hove net in $G$ with $\lim_{i\in I}\theta(KA_i)/\theta(A_i)=1$, it follows that \begin{displaymath}
		\lambda \, = \, \liminf_{i\in I}{f(A_i)}/{\theta(KA_i)} \, \leq \, c \, < \, \infty .
	\end{displaymath} 
	
	In particular, there is a cofinal subset $J\subseteq  I$  such that $(f(A_i))/\theta(A_{i}))_{i\in J}$ converges to $\lambda$. Thus, there exists $j\in J$ such that for all $i\in J$ with $i\geq j$ we have
	\begin{align}\label{form:invariancebound}
	\frac{f(A_i)}{\theta(A_i)} \, \leq \, \lambda+\epsilon.
	\end{align}
	Since also $\{i\in J \mid i\geq j\}$ is cofinal in $I$ we assume without lost of generality that (\ref{form:invariancebound}) holds for all $i\in J$.
	From Proposition \ref{pro:quasitiling} we obtain the existence of a
	finite subset $F\subseteq J$, $\delta>0$ as well as a compact, non-empty subset $D\subseteq G$ such that any $(\delta,D)$-invariant, compact subset $A\subseteq G$ can be $\epsilon$-quasi-tiled and the $\epsilon$-quasi-tiling centres can be chosen with the additional properties as in Proposition \ref{pro:quasitiling}.

	We now consider a compact and $(\delta,D)$-invariant subset $A\subseteq G$ and choose $\epsilon$-quasi-tiling centres $C_i$ such that these additional properties are satisfied, i.e.\ such that $\bigcup_{i\in F}A_iC_i\subseteq A$ and such that $R\defeq A\setminus \bigcup_{i\in F}A_iC_i$ satisfies $\theta(R)+\theta(\partial_KR)\leq \epsilon \theta(A)$. 
	From $F\subseteq J$ and (\ref{form:invariancebound}) we obtain $f(A_i)/\theta(A_i)\leq \lambda+\epsilon$ for every $i\in F$, and we compute
	\begin{align*}
	\frac{f\left(\bigcup_{i\in F} A_{i}C_i\right)}{\theta(A)}
	&\leq \sum_{i\in F}\sum_{g\in C_i}\frac{f(A_ig)}{\theta(A)}\\
	&=\sum_{i\in F}\sum_{g\in C_i}\frac{f(A_{i})}{\theta(A_{i})}\frac{\theta(A_{i})}{\theta(A)}\\
	&\leq (\lambda+\epsilon)\sum_{i\in F}\sum_{g\in C_i}\frac{\theta(A_{i})}{\theta(A)}.
	\end{align*}
	By the properties of $\epsilon$-quasi-tiling centres, $((A_{i}g)_{g \in C_{i}})_{i \in F}$ is an $\epsilon$-disjoint family. Therefore, $\bigcup_{i\in F}A_iC_i\subseteq A$ implies that
	\[\sum_{i\in F}\sum_{g\in C_i}\frac{\theta(A_{i})}{\theta(A)} \, = \, \sum_{i\in F}\sum_{g\in C_i}\frac{\theta(A_{i}g)}{\theta(A)}
	\, \leq \, \frac{1}{1-\epsilon}\frac{\theta(\bigcup_{i\in J} A_{i}C_i)}{\theta(A)} \, \leq \, \frac{1}{1-\epsilon}.\]
	We have shown that
	\begin{align}\label{form:fatset}
	\frac{f\left(\bigcup_{i\in F} A_{i}C_i\right)}{\theta(A)} \,
	&\leq \, \frac{\lambda+\epsilon}{1-\epsilon}.
	\end{align}
	Note that $e_G\in K$ implies 
	$K\overline{R}\subseteq R \cup \partial_KR. $
	As we require the $\epsilon$-quasi-tiling centres to satisfy $\theta(R)+\theta(\partial_KR)\leq \epsilon\theta(A)$, we obtain from the choice of the constant $c$ at the beginning of the proof that  
	\[f\left(\overline{R}\right) \, \leq \, c \theta(K\overline{R}) \, \leq \,
	c(\theta(R)+\theta(\partial_KR))
	\, \leq \, \epsilon c  \theta(A).\]
	Thus (\ref{form:fatset}) yields that
	\[\frac{f(A)}{\theta(A)} \, \leq \, \frac{f(\bigcup_{i\in F}A_iC_i)}{\theta(A)}+\frac{f\left(\overline{R}\right)}{\theta(A)} \, \leq \, \frac{\lambda+\epsilon}{1-\epsilon}+\epsilon c\]
	for every $(\delta,K)$-invariant, compact subset $A\subseteq G$. Hence, considering another van Hove net $\left(B_\iota\right)_{\iota\in \tilde{I}}$, we see that
	\[\limsup_{\iota \in \tilde{I}}\frac{f(B_\iota)}{\theta(B_\iota)} \, \leq \, \frac{\lambda+\epsilon}{1-\epsilon}+\epsilon c.\]  
	Since $\epsilon>0$ was arbitrary, this shows that, for any two van Hove nets $(A_i)_{i\in I}$ and $(B_\iota)_{\iota \in \tilde{I}}$ in $G$,
	\[\limsup_{\iota \in \tilde{I}}\frac{f(B_\iota)}{\theta(B_\iota)} \, \leq \, \lambda=\liminf_{i\in I}\frac{f(A_i)}{\theta(A_i)} \, < \, c,\]
	which clearly implies the statement of the theorem. 
\end{proof}

\section{Algebra isomorphism and topological models}
	
	As explained in the introduction we will need to find appropriate topological models for our measure preserving actions.
	
\subsection{Algebra isomorphism between measure preserving actions}
Let $(X,\mu)$ be a probability space and note that the relation given by $\mu(A\triangle B)=0$ is an equivalence relation on all subsets $A,B\subseteq X$. Let $[A]$ denote the equivalence class of $A\subseteq G$. The set  $\Sigma(X)$ of all equivalence classes is called the \emph{measure algebra} of $X$. Note that $[A\cap B]$ and $[A\cup B]$ are independent of the choice of representatives from $[A]$ and $[B]$ respectively. This allows to write $[A]\cap [B]$ and $[A]\cup [B]$ respectively for these sets. 
Furthermore, note that $\mu$ is constant on each equivalence class $[A]\in \Sigma(X)$, which allows to write $\mu[A]$. 
Consider now probability spaces $(X,\mu)$ and $(Y,\nu)$. 
We call a bijective mapping $\iota\colon \Sigma(X)\to \Sigma(Y)$, that satisfies $\iota([A]\cap [B])=\iota([A])\cap \iota([B])$, $\iota([A]\cup [B])=\iota([A])\cup \iota([B])$ 
and $\mu[A]=\nu(\iota[A])$ for all $[A],[B]\in \Sigma(X)$ an \emph{algebra isomorphism} between $(X,\mu)$ and $(Y,\nu)$.

If $\pi$ is a measure preserving action on $(X,\mu)$, then we may define $g.[A] \defeq [g.A]$ for all $[A]\in \Sigma(X)$ and $g\in G$. 
Given another measure preserving action $\phi$ of $G$ on a probability space $(Y,\nu)$, by an \emph{algebra isomorphism} between $\pi$ and $\phi$ we will mean an algebra isomorphism $\iota$ between $(X,\mu)$ and $(Y,\nu)$ such that $g.(\iota[A])=\iota(g.[A])$ for all $g\in G$ and $[A]\in \Sigma(A)$. Measure preserving actions $\pi$ and $\phi$ are called \emph{algebra isomorphic}, if there exists an algebra isomorphism between $\pi$ and $\phi$.

\subsection{About topological models}
For regular Borel probability measure $\nu$ on a compact Hausdorff space $K$ we define the \emph{support} $\operatorname{supp}(\nu)$ as the set of all $x\in K$ such that for any an open neighbourhood $U$ of $x$ we have $\nu(U)>0$. 
$\nu$ is said to have \emph{full support} if $K=\operatorname{supp}(\nu)$. 
Let $\pi$ be an measure preserving action of a topological group $G$ on a probability space $(X,\mu)$. A \emph{topological model} of $\pi$ is a continuous action $\phi$ of $G$ on a compact Hausdorff space $K$ together with an invariant regular Borel probability measure $\nu$ that has full support and such that $\phi$ considered as a measure preserving action on $(K,\nu)$ is algebra isomorphic to $\pi$.
\begin{remark}
	Note that we do not assume $(X,\mu)$ to be a Lebesgue space and that $K$ is not assumed to be separable or to be metrizable. Note furthermore that in contrast to the usual Jewett-Krieger theorem we do not require a topological model to be uniquely ergodic (see \cite{jewett1970prevalence, krieger1972unique} for the original work of R. Jewett and W. Krieger, 
	and \cite{weiss1985strictly, rosenthal1988strictly, eisner2015operator} for generalizations to actions of amenable discrete groups). 
	In the usual Jewett-Krieger theorem one also uses a stronger notion of isomorphism than algebra isomorphism. Nevertheless, for our purposes (entropy theory) algebra isomorphisms are sufficient and allow us to state our results beyond Lebesgue spaces. 
\end{remark}	

With similar arguments presented in \cite[Section 12.3]{eisner2015operator} one can show that any measure preserving action of a discrete group has a topological model. Since we are not aware of a reference beyond the context of discrete groups we include a full proof of the following. 

\begin{proposition}\label{pro:JKtype}
	Every measure preserving action $\pi$ of a topological group $G$ on a probability space $(X,\mu)$ admits a topological model. 
\end{proposition}	 
For the proof of Proposition \ref{pro:JKtype} we introduce the following notation from \cite{eisner2015operator}. 
Let $(X,\mu)$ and $(Y,\nu)$ be probability spaces and $p\in \{1,\infty\}$. 
A bijective positive operator $T\colon L^p(X,\mu)\to L^p(Y,\nu)$ is called a \emph{Markov isomorphism} if $T$ has a positive inverse and satisfies $T\chi_X=\chi_Y$ and $\int_Y Tfd\nu =\int_X f d\mu$ for all $f\in L^p(X,\mu)$. 
A Markov isomorphism $T\colon L^p(X,\mu)\to L^p(X,\mu)$ is called a \emph{Markov automorphism}. Let $\operatorname{MAut}^p(X,\mu)$ denote the set of all Markov automorphisms on $L^p(X,\mu)$ and equip $\operatorname{MAut}^p(X,\mu)$ with the strong operator topology. 
From \cite[Theorem 13.15]{eisner2015operator} we quote that $\operatorname{MAut}^p(X,\mu)$ is a topological group under composition.
From \cite[Proposition 13.6]{eisner2015operator} we furthermore know that the restriction mapping induces an isomorphism of topological groups (a  homeomorphism that is also a group homomorphism) between $\operatorname{MAut}^\infty(X,\mu)$ and $\operatorname{MAut}^1(X,\mu)$.

For a measure preserving action $\pi$ we denote $\pi_*\colon G\to \operatorname{MAut}^\infty(X,\mu)$ for the \emph{Koopman representation} defined by $\pi_*(g)(f) \defeq f\circ \pi^g$ for all $f\in L^\infty(X,\mu)$. Similarly we define the Koopman representation 
$\pi_*\colon G\to \operatorname{MAut}^1(X,\mu)$. 
For an action $\phi$ on a compact Hausdorff space $X$ for which each $\phi^g$ is continuous we furthermore define analogously the Koopman representation $\phi_*\colon G\to \mathcal{L}(C(X))$, where $\mathcal{L}(C(K))$ denotes the set of all bounded linear operators $C(K)\to C(K)$ equipped with the strong operator topology. From \cite[Theorem 4.17]{eisner2015operator} we quote that $\phi$ is a continuous action, if and only if the Koopman representation $\phi_*\colon G\to \mathcal{L}(C(X))$ is continuous.

\begin{proof}[Proof of Proposition \ref{pro:JKtype}]
	To construct the topological model $\phi$ and the algebra isomorphism $\iota$ for $\pi$ we essentially follow the arguments from \cite[Section 12.3]{eisner2015operator}. We will then add to these arguments and argue that $\phi$ is indeed a continuous action. (This is trivially satisfied if $G$ is discrete.) For the convenience of the reader we present the full construction.
	
	Note that $L^\infty(X,\mu)$ is a commutative $C^*$-algebra with unit $\chi_X$. Thus, by the Gelfand-Naimark theorem, there exists a compact Hausdorff space $K$ and a unital $C^*$-algebra isomorphism $\Phi\colon C(K)\to L^\infty(X,\mu)$. 
	By the Riesz--Markov--Kakutani representation theorem there exists a unique regular Borel probability measure $\nu$ on $K$ such that 
	$\int_K f d\nu=\int_X\Phi (f) d\mu$ for all $f\in C(K)$.
	It is straightforward to observe that the canonical operator $C(K)\to L^\infty(K,\nu)$ is a homeomorphism and that $\Phi$ preserves the respective $L^1$-norms. In particular, $\nu$ has full support. 
	With standard arguments one then obtains that $\Phi$ extends uniquely to a Markov isomorphism $L^1(K,\nu)\to L^1(X,\mu)$, which we also denote by $\Phi$. The restriction of $\Phi$ to the set of all characteristic functions in the corresponding $L^1$-spaces then induces an algebra isomorphism $\iota\colon \Sigma(K)\to \Sigma(X)$. 
	
	For $g\in G$ we consider the Markov automorphism $\Phi^{-1}\circ \pi^g_*\circ \Phi\colon C(K)\to C(K)$. From \cite[Theorem 7.23]{eisner2015operator} we know that there exists a unique homeomorphism $\phi^g$ such that 
	$\Phi^{-1}\circ \pi^g_*\circ \Phi=\phi^g_*$. It is straightforward to check that this defines an action $\phi\colon G\times K\to K$, that $\nu$ is invariant and that $\iota$ establishes an algebra isomorphism between the measure preserving action $\phi$ of $G$ (equipped with the discrete topology) on $(K,\nu)$ and $\pi$.
	
	It remains to show that $\phi$ is continuous with respect to the original topology of $G$. 
	Since $\pi$ is measure preserving, any characteristic function $f\in L^1(X,\mu)$ satisfies \begin{displaymath}
		\|\pi^g_*f-f\|_1 \, \longrightarrow \, 0 \, \text{ as } \, g\to e_G.
	\end{displaymath} By a straightforward density argument, this readily entails the continuity of the Koopman representation $\pi_*\colon G\to \operatorname{MAut}^1(X,\mu)$. 
	Since restriction induces a homeomorphism between 
	$\operatorname{MAut}^1(X,\mu)$ and $\operatorname{MAut}^\infty(X,\mu)$, the Koopman representation $\pi_*\colon G\to \operatorname{MAut}^\infty(X,\mu)$ is continuous, too. 
	For every $f\in C(K)$, we know that $\Phi(f)\in L^\infty(X,\nu)$ and hence \begin{displaymath}
		(\pi^g_*\circ \Phi)(f) \, = \, (\pi_*^g)(\Phi(f)) \, \longrightarrow \, \Phi(f) \, \text{ as } \, g\to e_G .
	\end{displaymath} Thus, $\phi_*^g(f)=(\Phi^{-1}\circ \pi^g_*\circ \Phi)(f)\rightarrow f$ as $g\to e_G$. This proves the continuity of the Koopman representation $\phi_*\colon G\to \mathcal{L}(C(K))$, which implies the continuity of $\phi$. 
\end{proof}

\section{Entropy and topological pressure}
\label{sec:entropyandpressure}

Let $\pi$ be an action of a group $G$ on a set $X$, and let $\mathcal{U}$ be any finite set of subsets of $X$. For $g\in G$, we define $\mathcal{U}_g\defeq\{g^{-1}.U \mid U\in \mathcal{U}\}$. For a finite set $F\subseteq G$, we furthermore define 
\[\left.\mathcal{U}_F \, \defeq \, \left\{\bigcap\nolimits_{g\in F} g^{-1}.U(g) \, \right\vert U\in \mathcal{U}^F\right\}
.\]
\subsection{Naive entropy and topological pressure}
\label{sub:naiveentropy}

	We next study naive entropy and naive topological pressure as a tool in order to gain insights on the Ollagnier entropy. 
	Let $G$ be a topological group.  
	For a continously measure preserving action $\pi$ of $G$ on a probability space $(X,\mu)$ we use the notation defined in the introduction to define the \emph{naive entropy}
	\[\operatorname{E}_\mu^{(nv)}(\pi) \, \defeq \, \sup_{\alpha}\inf_{F\in \mathcal{F}_+(G)}\frac{H_\mu(\alpha_F)}{|F|},\]
	where the supremum is taken over all finite measurable partitions $\alpha$ of $X$. For any continuous action $\pi$ of $G$ on a compact Hausdorff space $X$, any $f\in C(X)$ and any $F\in \mathcal{F}(G)$, we define $\Sigma_F f\defeq \sum_{g\in F}f\circ \pi^g$. 	
	The \emph{naive topological pressure} of $\pi$ at $f \in C(X)$ is defined as
	\[\operatorname{p}_f^{(nv)}(\pi) \, \defeq \, \sup_{\mathcal{U}}\inf_{F\in \mathcal{F}_+(G)}\frac{\operatorname{P}_{\Sigma_F f}(\mathcal{U}_F)}{|F|},\]
	where the supremum is taken over all finite open covers $\mathcal{U}$ of $X$. 	
	The \emph{naive topological entropy} of $\pi$ is $\operatorname{p}_0^{(nv)}(\pi)$.

	\subsubsection{A naive version of Goodwyn's theorem}

	The following version of Goodwyn's theorem generalizes \cite[Theorem 1.3]{burton2017naivenetropy}, where it is shown for actions of discrete groups and only naive topological entropy is considered. In order to prove this result we use a different technique following ideas of \cite{ollagnier1985ergodic}. 	
	We note that naive entropy does not depend on the choice of the group topology. Thus, in the case of naive topological entropy ($f=0$), the following can also directly be obtained from \cite[Theorem 1.3]{burton2017naivenetropy}. 
	
	\begin{theorem}[Goodwyn's theorem for naive entropy and topological pressure]
	\label{the:Goodwynnaivenetropy}
		Let $\pi$ be a continuous action of a topological group $G$ on a compact Hausdorff space $X$. If $\mu$ is a $\pi$-invariant regular Borel probability measure on $X$ and $f\in C(X)$, then
	\[\operatorname{E}_\mu^{(nv)}(\pi)+\mu(f)\, \leq \, \operatorname{p}_f^{(nv)}(\pi).\]
	\end{theorem}
	
	We will need the following notions for the proof of Theorem \ref{the:Goodwynnaivenetropy}.
	For finite measurable partitions $\alpha$ and $\beta$ of a probability space $(X,\mu)$ we define the \emph{entropy of $\alpha$ given $\beta$} as
	\[H_\mu(\alpha|\beta):=-\sum_{A\in \alpha,B\in \beta} \mu(A\cap B)\log\left(\frac{\mu(A\cap B)}{\mu(B)}\right).\]
	For further properties of this concept the reader is referred to 
	\cite[Section 4.3]{walters1982introduction}.
	Now let $X$ be a compact Hausdorff space.  
	  A finite measurable partition $\alpha$ of $X$ is said to be \emph{adapted} to an open cover $\mathcal{U}$ of $X$ if there exists an injective mapping $\mathfrak{U}\colon \alpha\to \mathcal{U}$ such that $A\subseteq \mathfrak{U}(A)$ for each $A\in \alpha$. For a regular Borel probability measure $\mu$ on $X$ and an open cover $\mathcal{U}$ on $X$, we define the \emph{overlap ratio} as\label{sym:overlapratio}
	$\operatorname{R}_\mu(\mathcal{U}) \defeq \sup_{\alpha,\beta}H_\mu(\alpha|\beta),$
	where the supremum is taken over all finite measurable partitions $\alpha$ and $\beta$ of $X$, which are adapted to $\mathcal{U}$. The overlap ratio satisfies the following simple invariance property. 
	If $G$ is a group acting by homeomorphisms on $X$, then it is straightforward to show that $\operatorname{R}_\mu(\mathcal{U})=\operatorname{R}_\mu(\mathcal{U}_g)$ for every $g\in G$. 
	For a more detailed account on these notions see \cite{ollagnier1985ergodic}.
	From the arguments presented in \cite[5.2.12]{ollagnier1985ergodic} we obtain furthermore the following.
	
\begin{lemma}\label{lem:overlapratiobound}
	Let $\mu$ be a regular Borel probability measure on a compact Hausdorff space $X$. 
	For every finite measurable partition $\alpha$ of $X$ and every $\epsilon>0$, there exists an open cover $\mathcal{U}$ of $X$ such that $\alpha$ is adapted to $\mathcal{U}$ and such that $\operatorname{R}_\mu(\mathcal{U})\leq \epsilon$. 
\end{lemma}

	We will furthermore need the following combinatorial lemma from  \cite{walters1982introduction}. 

\begin{lemma}\label{lem:exprefpimadness}
	Let $k\in \mathbb{N}$ and $a_1,\cdots, a_k,p_1,\cdots,p_k$ be given real numbers such that $p_i\geq 0$ and $\sum_{i=1}^k p_i=1$. Then 
	$\sum_{i=1}^k p_i(a_i-\log(p_i))\leq \log\left(\sum_{i=1}^ke^{a_i}\right)$.
\end{lemma}

	The following lemma will be useful for the proofs of the Theorems \ref{the:Goodwynnaivenetropy} and \ref{the:OW:Goodwyns} below. 

\begin{lemma}\label{lem:Goodwynsmainlemma}
	Let $\pi$ be a continuous action of a topological group $G$ on a compact Hausdorff space $X$. Let $\mu$ be a $\pi$-invariant, regular Borel probability measure on $X$. 
	For every finite measurable partition $\alpha$ of $X$ and every $\epsilon>0$, there exists a finite open cover $\mathcal{U}$ such that, for any $f\in C(X)$ and any finite set $F\subseteq G$, \begin{displaymath}
		H_\mu(\alpha_F)+\mu(f) \, \leq \, \operatorname{P}_f(\mathcal{U}_F)+|F|\epsilon.
	\end{displaymath}
\end{lemma}
\begin{proof}
	From Lemma \ref{lem:overlapratiobound} we obtain the existence of an open cover $\mathcal{U}$ of $X$ such that $\alpha$ is adapted to $\mathcal{U}$ and such that $\operatorname{R}_\mu(\mathcal{U})\leq \epsilon$.	
	Consider now $f\in C(X)$ and $F\in \mathcal{F}_+(G)$. Consider furthermore a finite measurable partition $\beta$ of $X$ that is adapted to $\mathcal{U}_{F}$. Using Lemma \ref{lem:exprefpimadness} we compute
\begin{align*}
	H_\mu(\beta)+\mu(f)&\leq \sum_{B\in \beta}\mu(B)\left(\sup_{x\in B}f(x)-\log(\mu(B))\right) 
		\leq \log\left(\sum_{B\in \beta} e^{\sup_{x\in B}f(x)}\right)\\
		&=\log\sum_{B\in \beta} \sup_{x\in B}e^{f(x)}
		\leq \log\sum_{U\in \mathcal{U}_{F}} \sup_{x\in U}e^{f(x)}
		=\Pcov_{f}(\mathcal{U}_{F}).
\end{align*} 
	Now consider $g\in F$ and recall that $\mathcal{U}_g\preceq \mathcal{U}_F\preceq \beta$. Thus for each $B\in \beta$ there exists $U_B\in \mathcal{U}_g$ such that $B\subseteq U_B$. Considering \[ \left. \gamma^{(g)} \defeq \left\{\bigcup \{ B \in \beta \mid U_B=U \} \, \right\vert U\in \mathcal{U}_g\right\}\setminus \{\emptyset\}\] we obtain a finite measurable partition of $X$ that is adapted to $\mathcal{U}_g$ and satisifes $\gamma^{(g)}\preceq \beta$. The finite measurable partition $\alpha_g$ is also adapted to $\mathcal{U}_g$. Thus the basic properties of the static entropy and the invariance of the overlap ratio allow to estimate	\begin{align*}
H_{\mu}(\alpha_{F}|\beta)
	\leq \sum_{g\in F}H_{\mu}(\alpha_{g}|\beta)
	\leq \sum_{g\in F}H_{\mu}(\alpha_{g}|\gamma^{(g)})
	\leq \sum_{g\in F}\operatorname{R}_\mu(\mathcal{U}_{g})
	\leq |F|\operatorname{R}_\mu(\mathcal{U})
	\leq |F|\epsilon.
	\end{align*}
Combining our observations we thus obtain
	\begin{displaymath}
H_\mu(\alpha_{F})+\mu(f)
	\leq H_\mu(\beta)+\mu(f)+H_{\mu}(\alpha_{F}|\beta)
	\leq \Pcov_{f}(\mathcal{U}_{F})+|F|\epsilon. \qedhere
	\end{displaymath}
\end{proof}

	\begin{proof}[Proof of Theorem \ref{the:Goodwynnaivenetropy}]
	Of course, it suffices to verify that $\operatorname{E}_\mu^{(nv)}(\pi)+\mu(f)\leq \operatorname{p}_f^{(nv)}(\pi)+\epsilon$ for every $\epsilon> 0$. For this purpose, let $\epsilon>0$. Consider any finite measurable partition $\alpha$ of $X$. By Lemma \ref{lem:Goodwynsmainlemma}, there exists a finite open cover $\mathcal{U}$ such that, for every $F\in \mathcal{F}(G)$, 
	\[H_\mu(\alpha_{F})+\mu\left(\sum\nolimits_{F}f\right) \, \leq \, \operatorname{P}_{\sum_{F}f}(\mathcal{U}_{F})+|F|\epsilon \]
	where we let $\Sigma_F f\defeq \sum_{g\in F}f\circ \pi^g$.
	In particular,
	\begin{align*}
	\inf_{F\in \mathcal{F}_+(G)}\frac{H_\mu(\alpha_{F})}{|F|}+\mu(f) \,
	&= \, \inf_{F\in \mathcal{F}_+(G)}\frac{H_\mu(\alpha_{F})+\mu(\sum_{F}{f})}{|F|}\\
	&\leq \, \inf_{F\in \mathcal{F}_+(G)} \frac{\operatorname{P}_{\sum_{F}{f}}(\mathcal{U}_{F})}{|F|}+\epsilon\\
	&\leq \, \operatorname{p}_f^{(nv)}(\pi)+\epsilon.
	\end{align*}
	Taking the supremum over all finite measurable partitions of $X$, we conclude that 
	$\operatorname{E}_\mu^{(nv)}(\pi)+\mu(f)\leq \operatorname{p}_f^{(nv)}(\pi)+\epsilon$, as desired.
\end{proof}
\begin{remark}
	It remains open whether the variational principle holds for naive entropy and naive topological pressure (in case there exist invariant and regular Borel probability measures $\mu$ on $X$).	
	 However, if such measures exist and $G$ is metrizable or locally compact, then Theorem \ref{the:naivetopologicalpressureiszero} below readily entails that
	$\sup_{\mu}\operatorname{E}_\mu^{(nv)}(\pi)= \operatorname{E}^{(nv)}(\pi) \, (=0)$ as soon as $G$ is non-discrete. 
\end{remark}

\begin{remark}\label{rem:Ollagnierserror}
	In \cite{ollagnier1982variational} and in \cite[5.2.12]{ollagnier1985ergodic} the proof of Goodwyn's half of the variational principle makes heavy use of subadditivity properties of the overlap ratio \cite[Proposition 5.2.11]{ollagnier1985ergodic} and the following claim. For $\mu\in \mathcal{M}_G(X)$ and all finite open covers $\mathcal{U}$ it is claimed that whenever $\alpha$ is a finite measurable partition of $X$ that is adapted to $\mathcal{U}$ via an injective map $\mathfrak{U}\colon \alpha\to \mathcal{U}$ such that $\mu(\mathfrak{U}(A)\setminus A)$ is small, then $\alpha_F$ is adapted to $\mathcal{U}_F$ for any finite set $F\subseteq G$. Considering $X=\{1,2,3\}$ and the action of $\mathbb{Z}$ on $X$ introduced by $\pi^1(1)=3$, $\pi^1(2)=2$ and $\pi^1(3)=1$ we can choose the partition 
	$\alpha=\{\{1,2\},\{3\}\}$ that is adapted to 
	$\mathcal{U}=\{X,\{1,3\}\}$ via $\mathfrak{U}\colon \alpha\to \mathcal{U}$ that sends $\{1,2\}\mapsto X$ and $\{3\}\mapsto \{1,3\}$. The Dirac measure $\delta_2$ is then an invariant and regular Borel probability measure that satisfies $\delta_2(\mathfrak{U}(A)\setminus A)=0$ for all $A\in \alpha$. 
	Nevertheless, $\alpha_{\{0,1\}}=\{\{1\},\{2\},\{3\},\emptyset\}$ and $\mathcal{U}_{\{0,1\}}=\mathcal{U}$ and thus clearly $\alpha_{\{0,1\}}$ is not adapted to $\mathcal{U}_{\{0,1\}}$. 	
	In a correspondence with J.\ M.\ Ollagnier it was discussed how to repair the proof of the statement in the context of actions of discrete amenable groups and the proof above contains ideas from this discussion.
\end{remark}

	\subsubsection{Naive topological entropy for actions of non-discrete groups}
	\label{subsub:naivenon-discrete}
	It is well known that naive (measure-theoretic) and topological entropy can only take the values $0$ and $\infty$ for actions of non-amenable discrete groups \cite{burton2017naivenetropy}. Our next objective is to show that naive entropy is $0$ whenever the acting topological group is non-discrete and, moreover, metrizable or locally compact.

\begin{theorem}\label{the:naivetopologicalpressureiszero}
	For any continuous action $\pi$ of a topological group that contains an infinite precompact subset we have $\operatorname{p}_0^{(nv)}(\pi)=0$.
\end{theorem}
	Since any measure preserving action has a topological model by Proposition~\ref{pro:JKtype}, the combination of Theorem \ref{the:Goodwynnaivenetropy} and Theorem~\ref{the:naivetopologicalpressureiszero} readily entails the following. 
\begin{corollary}\label{cor:naiveentropyiszero}
	For any measure preserving action $\pi$ of a topological group that contains an infinite precompact subset we have $\operatorname{E}_\mu^{(nv)}(\pi)=0$.
\end{corollary}
	
\begin{remark}
	Let $G$ be a non-discrete topological group. If $G$ is metrizable or locally compact, then $G$ has an infinite compact subset. 
	Indeed, if $G$ is metrizable, then $G$ admits a non-stationary convergent sequence $(g_n)_{n\in \mathbb{N}}$, and we observe that $\{g_n \mid n\in \mathbb{N}\}$ constitutes an infinite compact subset of $G$ \cite{levine1968equivalence}.  	
	If $G$ is locally compact, then $G$ contains a compact subset with non-empty interior, which is necessarily infinite by non-discreteness of $G$. 
\end{remark}

	In order to prove Theorem \ref{the:naivetopologicalpressureiszero} we need the following notation from \cite{ollagnier1982variational}. Consider a continuous action of a topological group $G$ on a compact Hausdorff space $X$. 
	For $\eta\in \mathbb{U}_X$ and $A\subseteq G$, we define $\eta_A \defeq \cap_{g\in A}\eta_g$, where \begin{displaymath}
		\eta_g \, \defeq \, \left. \! \left\{(x,y)\in X^2 \, \right\vert (g.x,g.y)\in \eta \right\} \qquad (g \in G) .
	\end{displaymath} 
	Conveniently modifying an argument from \cite[Lemma 4.2]{hauser2020relative}, we present a short proof of the following statement. 
	
	\begin{lemma}
	Let $G$ be a topological group acting continuously on a compact Hausdorff space $X$. For $\eta\in \mathbb{U}_X$ and $A\subseteq G$ precompact we have $\eta_A\in \mathbb{U}_X$. 
	\end{lemma}
	\begin{proof}
		Consider a symmetric $\kappa\in \mathbb{U}_X$ such that $\kappa\kappa\kappa\subseteq \eta$.
		Since $\pi$ is continuous and $X$ is compact, a standard argument 
		shows that there exists an identity neighbourhood $V \in \mathcal{U}(G)$ such that $\{ (g.x,x) \mid g \in G, \, x \in X \} \subseteq \kappa$.
By $A$ being precompact, there is a finite subset $F\subseteq G$ with $A\subseteq VF$. A straightforward computation now shows that $\eta_A\supseteq \kappa_F\in \mathbb{U}_X$.
	\end{proof}
	
	Consider a continuous action of a topological group $G$ on a compact Hausdorff space $X$. 
	For open covers $\mathcal{U}$ and $\mathcal{V}$ of $X$ and a subset $A\subseteq G$, we say that $\mathcal{V}$ $A$-\emph{refines} $\mathcal{U}$ if, for each $g\in A$, the cover $\mathcal{V}$ is finer than $\mathcal{U}_g$. 
	From \cite[Lemma 3.2]{schneider2015topological} we quote the following for which we  include a short (alternative) proof. 
\begin{lemma}\label{lem:nv:Arefining}
Let $G$ be a topological group acting continuously on a compact Hausdorff space $X$. If $\mathcal{U}$ is a finite open cover of $X$ and $A\subseteq G$ is precompact, then there exists a finite open cover $\mathcal{V}$ of $X$ which $A$-refines $\mathcal{U}$.
\end{lemma}
\begin{proof}
	Let $\eta$ be a Lebesgue entourage of $\mathcal{U}$.
	Since $\eta_A\in \mathbb{U}_X$ for $x\in X$ there exists an open neighbourhood $V_x$ of $x$ that is contained in $\eta_A[x]$. 
	For $g\in A$ we observe the existence of $U\in \mathcal{U}_g$ such that $V_x\subseteq \eta_g[x]\subseteq U$ and hence that $\{V_x \mid x\in X\}$ $A$-refines $\mathcal{U}$. 
	Since $X$ is compact we can choose a finite subcover $\mathcal{V}$ of  $\{V_x \mid x\in X\}$. 
\end{proof}
\begin{proof}[Proof of Theorem \ref{the:naivetopologicalpressureiszero}]
	Let $\mathcal{U}$ be a finite open cover of $X$ and consider an infinite and precompact set $A$. Let $\mathcal{V}$ be a finite open cover that $A$-refines $\mathcal{U}$. In particular we obtain that $\mathcal{V}$ is finer than $\mathcal{U}_F$ for all finite subsets $F\subseteq A$ and thus 
	$P_0(\mathcal{U}_F)
	=\log |\mathcal{U}_F|
	\leq |\mathcal{V}|<\infty.$
	Since $A$ is assumed to be infinite, we observe 
	\[\inf_{F\in \mathcal{F}_+(G)} \frac{\operatorname{P}_0(\mathcal{U}_F)}{|F|} \, = \, 0. \]
	Taking the supremum over all finite open covers of $X$, we obtain the statement. 
\end{proof}

We conclude this subsection with an example of a non-precompact topological group that does not contain an infinite precompact subset, which has been communicated to the authors by Maxime Gheysens.

\begin{example}\label{example:maxime} Let $H$ by any non-trivial discrete group (such as, for example, the additive group of integers). Furthermore, consider any
uncountable set $I$. We turn the direct power $G \defeq H^{I}$ into a topological group by endowing it with the \emph{countable-box
topology}, i.e.~the one generated by the basic open subsets of the form
\begin{displaymath}
	\{ g \in G \mid \forall i \in C \colon \, g_{i} \in U_{i} \} ,
\end{displaymath} where $C$ is any countable subset of $I$ and $(U_{i})_{i \in C}$ is any family of subsets of $H$. It is easy to verify that this topology indeed constitutes a group topology on $G$. Moreover, this topology is non-discrete (for cardinality reasons) and Hausdorff (as it contains the usual product topology). Now, the topological group $G$ has the additional property that each of its countable subsets is uniformly discrete. Indeed,
if $A$ is a countable subset of $G$, then for any pair $(g,h) \in A \times A$ with $g \neq h$, we can find an index $i(g, h) \in I$ such that $g_{i(g,h)} \ne h_{i(g,h)}$, so that $A$ will be $U$-discrete for the open identity neighborhood \begin{displaymath}
	U \, \defeq \, \{ g \in G \mid \forall i \in C \colon \, g_{i} \in \{ e_{H} \} \} \, \in \, \mathcal{U}(G) ,
\end{displaymath} where $C \defeq \{ i(g,h) \mid g,h \in A, \, g\ne h \}$.
Consequently, $G$ cannot posses any countably infinite precompact subset, which -- by the axiom of choice -- entails that $G$ does not contain any infinite precompact subset at all. Choosing $H$ to be abelian, the group $G$ will be abelian, too, thus amenable even as a discrete group. \end{example}

\subsection{Ollagnier entropy}
\label{sub:Ollagnierentropy}
	Let $\pi$ be a measure preserving action of an amenable topological group $G$ on a probability space $(X,\mu)$.
	Recall from the introduction that we define \emph{Ollagnier entropy} as	
	\[\operatorname{E}_\mu^{(Oll)}(\pi) \, \defeq \, \sup_{\alpha}\lim_{i\in I}\frac{H_\mu(\alpha_{F_i})}{|F_i|},\]
	where the supremum is taken over all finite measurable partitions of $X$ and where $(F_i)_{i\in I}$ is a thin F\o lner net. 
	Here the limit exists and is independent from the choice of a thin F\o lner net by Ollagnier's Lemma (Theorem \ref{the:ollagnierslemma}), which can be applied by the following lemma. 
	
	\begin{lemma}\label{lem:continuouslympandEntropy}
	Let $\pi$ be a measure preserving action of a topological group $G$ on a probability space $(X,\mu)$. 
		For any finite measurable partition $\alpha$ of $X$, the map \begin{displaymath}
			\mathcal{F}_+(G) \, \longrightarrow \, \mathbb{R}_{\geq 0}, \quad F \, \longmapsto \, H_\mu(\alpha_F)
		\end{displaymath} is right-invariant, continuous and satisfies Shearer's inequality.  
	\end{lemma}
	\begin{proof} Let $\alpha$ be a finite measurable partition of $X$. It is standard to show that the considered map is right-invariant and satisfies Shearer's inequality (see, e.g.,~\cite{downarowicz2015shearer}). 
	To show the continuity, consider any $F\in \mathcal{F}_+(G)$. 
	Since $\pi$ is measure preserving, there exists 
	an open neighbourhood $V$ of $e_G$ such that 
	$\mu(A\triangle g.A)<\epsilon/(|F||\alpha^F|)$ for each $A\in \alpha$. 
	Now, let $E\in V[F]$ and choose a bijection $b\colon F\to E$ such that $b(g) \in Vg$ for each $g\in F$. 
	For any map $A^{(\cdot)}\colon F\to \alpha, \, g\mapsto A^{(g)}$, we observe that
	\begin{align*}
	\mu\!\left( \! \left(\bigcap\nolimits_{g\in F}g^{-1}.A^{(g)}\right) \triangle \left(\bigcap\nolimits_{g\in F}(b(g))^{-1}.A^{(g)}\right) \! \right) \,
	&\leq \, \mu \! \left(\bigcup\nolimits_{g\in F}g^{-1}.A^{(g)}\triangle (b(g))^{-1}.A^{(g)} \right)\\
	&\leq \, \sum\nolimits_{g\in F} \mu \! \left(g^{-1}.A^{(g)}\triangle (b(g))^{-1}.A^{(g)} \right)\\ 
	&= \, \sum\nolimits_{g\in F} \mu\!\left((b(g)g^{-1}).A^{(g)}\triangle A^{(g)} \right)\\
	&\leq \, \frac{|F|\epsilon}{|F||\alpha^F|} \, = \, \frac{\epsilon}{|\alpha^F|}.
	\end{align*}
	This shows that  
	\[\mathcal{F}_+(G) \, \longrightarrow \, \mathbb{R}_{\geq 0}, \quad F \, \longmapsto \, \sum\nolimits_{A^{(\cdot)}\in \alpha^F}\mu\!\left(\!\left(\bigcap\nolimits_{g\in F}g^{-1}.A^{(g)}\right) \triangle \left(\bigcap\nolimits_{g\in F}(b(g))^{-1}.A^{(g)}\right)\!\right)\]
	is continuous. 
	By a standard argument (as for example contained in \cite[Lemma 4.15]{walters1982introduction}), we thus obtain the continuity of $\mathcal{F}_+(G) \to \mathbb{R}_{\geq 0} , \, F\mapsto H_\mu(\alpha_F)$. 
	\end{proof}
	As another application of Ollagnier's Lemma we obtain that the limit in the definition of the Ollagnier entropy is actually an infimum, i.e.\ the following. 	
	
	\begin{theorem}
		Any measure preserving action $\pi$ of an amenable topological group on a probability space $(X,\mu)$ satisfies 
	$\operatorname{E}_\mu^{(Oll)}(\pi)=\operatorname{E}_\mu^{(nv)}(\pi).$
	\end{theorem}
	
	From Corollary \ref{cor:naiveentropyiszero} we conclude  the following. 
	
	\begin{corollary}
		Let $\pi$ be a measure preserving action of an amenable topological group $G$ on a probability space $(X,\mu)$.  If $G$ contains an infinite precompact subset, then 
	$\operatorname{E}_\mu^{(Oll)}(\pi)=0.$
	\end{corollary}

\subsection{Ornstein-Weiss entropy}
\label{sub:OWentropy}

	Let $G$ be an amenable unimodular locally compact group and consider a van Hove net $\mathcal{A}=(A_i)_{i\in I}$ and a Delone set $\omega$ in $G$. 
	For a measure preserving action $\pi$ of $G$ on a probability space $(X,\mu)$ we define the \emph{Ornstein-Weiss entropy} as
	\[\operatorname{E}^{(OW)}_{\mu,\mathcal{A},\omega}(\pi)\defeq \sup_{\alpha}\limsup_{i\in I}\frac{H_{\mu}(\alpha_{A_i\cap \omega})}{\theta(A_i)},\]
	where the supremum is taken over all finite measurable partitions of $X$.  
	For a continuous action $\pi$ of $G$ on a compact Hausdorff space $X$ and $f\in C(X)$ we define the \emph{Ornstein-Weiss topological pressure} as
	\[\operatorname{p}_{f,\mathcal{A},\omega}^{(OW)}(\pi)\defeq \sup_{\mathcal{U}}\limsup_{i\in I}\frac{P_{f}(\mathcal{U}_{A_i\cap \omega})}{\theta(A_i)},\]
	where the supremum is taken over all finite open covers $\mathcal{U}$ of $X$. 
	We will next show that these notions are independent of the choice of a van Hove net $\mathcal{A}$ and a Delone set $\omega$. 
	
	\subsubsection{Ornstein-Weiss entropy}
	Consider a measure preserving action $\pi$ of an amenable unimodular locally compact group $G$ on a probability space $(X,\mu)$. 
	Note that Ornstein-Weiss entropy is an invariant under algebra isomorphism of measure preserving actions (for a fixed van Hove net and a fixed Delone set). Thus by considering a topological model of $\pi$ (see Proposition \ref{pro:JKtype}) we thus assume without lost of generality that $X$ is a compact Hausdorff space, that $\mu$ is a regular Borel probability measure on $X$ and that $\pi$ is a continuous action on $X$. 
	This topological setup will allow us to present an alternative approach to Ornstein-Weiss entropy that facilitates an application of the Ornstein-Weiss lemma. 
	
	For $\eta\in \mathbb{U}_X$ we define $H_\mu[\eta]\defeq \inf_{\alpha}H_\mu(\alpha)$, where the infimum is taken over all finite measurable partitions $\alpha$ of $X$ at scale $\eta$. 
		Now recall from Subsection \ref{subsub:naivenon-discrete} that we define $\eta_A\defeq \cap_{g\in A}\{(x,y)\in X^2 \mid (g,x,g.y)\in \eta\}$ 
		and that $\eta_A\in \mathbb{U}_X$ for any compact subset $A\subseteq G$. 		
		It is straightforward to show that $\mathcal{K}(G)\ni A\mapsto H_\mu[\eta_A]$ is monotone, right-invariant and subadditive. The Ornstein-Weiss lemma thus allows to define 
		\[\operatorname{E}_\mu^{(OW)}[\pi,\eta] \, \defeq \, \lim_{i\in I}\frac{H_\mu[\eta_{A_i}]}{\theta(A_i)}\]
		independently from the choice of the van Hove net.
		We will next show the following.
		\begin{theorem}\label{the:OW:entropyindependenceandrelation}
			Let $\pi$ be a continuous action of an amenable unimodular locally compact group on a compact Hausdorff space $X$ and $\mu$ be an invariant regular Borel probability measure on $X$.
	For any van Hove net $\mathcal{A}$ and any Delone set $\omega$ in $G$, we have
			\[\operatorname{E}_{\mu,\mathcal{A},\omega}^{(OW)}(\pi) \, = \, \sup_{\eta\in \mathbb{U}_X}\operatorname{E}_\mu^{(OW)}[\pi,\eta].\] 
	\end{theorem}

\begin{corollary}
	For any measure preserving action $\pi$ of an amenable unimodular locally compact group on a probability space $(X,\mu)$, the Ornstein-Weiss entropy \begin{displaymath}
		\operatorname{E}_\mu^{(OW)}(\pi)\, \defeq \, \operatorname{E}_{\mu,\mathcal{A},\omega}^{(OW)}(\pi)
	\end{displaymath} is independent of both the choice of a van Hove net and the choice of a Delone set.
\end{corollary}

\begin{remark}
	It would be interesting to clarify the relationship between our notion of Ornstein-Weiss entropy and the notions of \emph{spatial entropy} and \emph{orbital entropy} considered in \cite{ornstein1987entropy}. 
\end{remark}
				
		We begin the proof of Theorem \ref{the:OW:entropyindependenceandrelation} with the following lemma.
	
	\begin{lemma}\label{lem:OW:entropylemma1}
	Let $\mu$ be a regular Borel probability measure on a compact Hausdorff space~$X$. For $\epsilon>0$ and a finite measurable partition $\alpha$ of $X$ there exists $\eta\in \mathbb{U}_X$ such that $\alpha$ is at scale $\eta$ and such that $H_\mu(\alpha|\beta)<\epsilon$ for every finite measurable partition $\beta$ of $X$ at scale $\eta$. 
	\end{lemma}
	\begin{proof}
	We represent $\alpha \defeq \{A_1,\dots, A_r\}$. 
	It is well known that there exists $\delta>0$ such that whenever 
	$\gamma=\{C_1,\dots,C_r\}$ 
	is a partition that satisfies $\sum_{i=1}^r \mu(A_i\triangle C_i)<\delta$, then $H_\mu(\alpha|\gamma)<\epsilon$ \cite[Lemma 4.15]{walters1982introduction}). 
	As $\mu$ is regular there exist compact subsets $D_i\subseteq A_i$ with 
	$\mu(A_i\setminus D_i)\leq \delta/(2r^2)$ for $i\in \{1,\dots,r\}$. Let 
	$D_0\defeq X\setminus \bigcup_{i=1}^r D_i$, $U_i\defeq D_0\cup D_i$, $\eta\defeq \bigcup_{i=1}^r (U_i\times U_i)$ and $\mathcal{U}\defeq \{U_1,\dots,U_r\}$. 
	Clearly $\alpha$ is at scale $\eta$. 
	
	Consider a finite measurable partition $\beta$ at scale $\eta$.
	Let $B\in \beta$. Whenever $B\subseteq D_0$, then $B\subseteq U_i$ for any $i\in \{1,\dots, r\}$. Otherwise, there exists $i\in \{1,\dots,r\}$ and $d\in B$ with $d\in D_i$. 
	In particular, we observe $d\notin U_j$ for $j\in\{1,\dots, r\}$ with $j\neq i$. 
	For $b\in B$ we have 
	$(b,d)\in B^2\subseteq \eta=\bigcup_{i=1}^rU_i$ and hence $b\in U_i$. This shows 
	$B\subseteq U_i$ and we have shown that $\beta$ is finer than $\mathcal{U}$. 	
	In particular, there exists a finite measurable partition $\gamma=\{C_1,\dots,C_r\}$ of $X$ with $C_i\subseteq U_i$ and $\mathcal{U}\preceq \gamma \preceq \beta$.
	Since $\gamma$ is a partition we obtain $D_i\subseteq C_i$ for all $i\in \{1,\dots, r\}$ and compute $C_i\triangle A_i\subseteq (U_i\setminus D_i)\cup (A_i\setminus D_i)=D_0\cup(A_i\setminus D_i)$ and hence 
	\begin{displaymath}
		\mu(C_i\triangle A_i) \, \leq \, \mu(D_0)+\mu(A_i\setminus D_i) \, \leq \, \delta/(2r)+\delta/(2r^2) \, \leq \, \delta/r .
	\end{displaymath} Our choice of $\delta$ thus implies $H_\mu(\alpha|\beta)\leq H_\mu(\alpha|\gamma)\leq \epsilon$. 	
	\end{proof}		
	\begin{remark}
	Consider $\mathcal{U}$ as constructed in the previous proof. 
	It is standard to show that any finite measurable partition $\beta$ of $X$ that is finer than 
	$\mathcal{U}$ satisfies $H_\mu(\alpha|\beta)\leq \epsilon$. For example, this argument is carried out in the proof of \cite[Theorem 3.5]{huang2011local}. Nevertheless, we decided to include the short proof for the reader's convenience. 
	\end{remark}

	\begin{lemma}\label{lem:vanHovenets} Let $G$ be an amenable unimodular locally compact group. 
	Let $\omega \subseteq G$ be closed and let $K\subseteq G$ be compact such that $K\omega=G$.
	Let $(A_i)_{i\in I}$ be a van Hove net in $G$ and define
	$F_i \defeq \omega\cap A_i$ for each $i\in I$. 
	Then, the net $(KF_i)_{i\in I}$ is van Hove and satisfies $\lim_{i\in I}\theta(KF_i)/\theta(A_i)=1$ and \begin{displaymath}
		\lim_{i\in I}{\theta(KF_i \triangle A_i)}/{\theta(A_i)} \, = \, \lim_{i\in I}{\theta(KF_i \triangle A_i)}/{\theta(KF_i)} \, = \, 0 .
	\end{displaymath}	   
\end{lemma}	
\begin{proof}
	From Lemma \ref{lem:pre:vanHovenets} we obtain the existence of a van Hove net $(B_i)_{i\in I}$ that satisfies $K^{-1}B_i\subseteq A_i$ and $B_i^c\subseteq K A_i^c$ for all $i\in I$ and furthermore $\lim_{i\in I}\theta(B_i)/\theta(A_i)=1$. It is straightforward to show that $B_i\subseteq KF_i$. We compute 
	\[1\leftarrow \frac{\theta(B_i)}{\theta(A_i)}\leq \frac{\theta(KF_i)}{\theta(A_i)}\leq \frac{\theta(KA_i)}{\theta(A_i)}\rightarrow 1\]
	and observe $\lim_{i\in I}\theta(KF_i)/\theta(A_i)=1$.
	We furthermore obtain $(KF_i)^c\subseteq B_i^c\subseteq KA_i^c$. For a compact subset $C\subseteq G$ this allows to compute
	\[\partial_C (KF_i)=CKF_i \cap C\overline{(KF_i)^c}\subseteq CKA_i\cap CK\overline{A_i^c}=\partial_{CK}A_i.\]
	We deduce that $(KF_i)_{i\in I}$ is a van Hove net. 	
	To show $\lim_{i\in I}\theta(KF_i \triangle A_i)/\theta(A_i)=0$ let $k\in K$. As $\mathcal{P}(X)$ is an abelian group under $\triangle$, with the identity as inverse map and neutral element $\emptyset$, we compute 
\begin{align*}
	KF_i\triangle A_i
	&=KF_i \triangle \emptyset \triangle A_i
	=(KF_i \triangle kA_i)\triangle (kA_i \triangle A_i)
	\subseteq (KF_i \triangle kA_i)\cup (kA_i \triangle A_i)\\
	&\subseteq (KA_i\setminus kA_i)\cup (kA_i\setminus KF_i) \cup (kA_i \triangle A_i)\\
	&\subseteq 
	\partial_{K}A_i
	\cup (kA_i\setminus KF_i) \cup \partial_{\{e_G,k\}} A_i\cup \partial_{\{e_G,k^{-1}\}} A_i.
\end{align*}
Now recall that $B_i\subseteq KF_i$ and $K^{-1}B_i\subseteq A_i$. Hence $B_i\subseteq kA_i$ and we obtain
\[0\leq \frac{\theta(kA_i\setminus KF_i)}{\theta(A_i)}\leq \frac{\theta(kA_i\setminus B_i)}{\theta(A_i)}=\frac{\theta(kA_i)}{\theta(A_i)}-\frac{\theta(B_i)}{\theta(A_i)}\rightarrow 1-1=0.\]
Since $(A_i)_{i\in I}$ is van Hove, we conclude
\begin{align*}
	0\leq \frac{\theta(KF_i\triangle A_i)}{\theta(A_i)}
	\leq \frac{\theta(\partial_K A_i)}{\theta(A_i)}+\frac{\theta(kA_i\setminus KF_i)}{\theta(A_i)}+
	\frac{\theta(\partial_{\{e_G,k\}} A_i)}{\theta(A_i)}
	+\frac{\theta(\partial_{\{e_G,k^{-1}\}} A_i)}{\theta(A_i)}
	\rightarrow 0,
\end{align*} as desired and furthermore, $\lim_{i\in I}{\theta(KF_i \triangle A_i)}/{\theta(KF_i)}=0$. 
\end{proof}

	\begin{proof}[Proof of Theorem \ref{the:OW:entropyindependenceandrelation}]
	Let $\mathcal{A}=(A_i)_{i\in I}$ be a van Hove net and $\omega$ be a Delone set in $G$. 
	Let $K\subseteq G$ be a compact subset such that $K\omega=G$ and $e_G\in K$. Let $V$ be a compact neighbourhood of $e_G$ such that $(Vv)_{v\in \omega}$ is a disjoint family.
	We abbreviate $F_i \defeq A_i\cap \omega$ and observe 
	$\eta_{F_i}\supseteq \eta_{KF_i}=(\eta_K)_{F_i}$. 
	Hence $H_\mu[\eta_{F_i}]\leq  H_\mu[\eta_{KF_i}]=H_\mu[(\eta_K)_{F_i}]$. 
	Since $\eta_K\in \mathbb{U}_X$ we observe 
	\[\limsup_{i\in I}\frac{H_\mu[\eta_{F_i}]}{\theta(A_i)}
	\, \leq \, \limsup_{i\in I}\frac{H_\mu[\eta_{KF_i}]}{\theta(A_i)}
	\, \leq \, \sup_{\epsilon\in \mathbb{U}_X} \limsup_{i\in I}\frac{H_\mu[\epsilon_{F_i}]}{\theta(A_i)}.\]
	Taking the supremum over all $\eta\in \mathbb{U}_X$ we obtain from Lemma \ref{lem:vanHovenets} that
	\begin{align*}
	\sup_{\eta\in \mathbb{U}_X} \operatorname{E}_\mu^{(OW)}[\pi,\eta]
	\, &= \, \sup_{\eta\in \mathbb{U}_X} \limsup_{i\in I}\frac{H_\mu[\eta_{A_i}]}{\theta(A_i)} \\
	& = \, \sup_{\eta\in \mathbb{U}_X} \limsup_{i\in I}\frac{H_\mu[\eta_{KF_i}]}{\theta(A_i)}
	\, = \, \sup_{\eta\in \mathbb{U}_X} \limsup_{i\in I}\frac{H_\mu[\eta_{F_i}]}{\theta(A_i)} .
	\end{align*}
	Now consider $\eta\in \mathbb{U}_X$ and a finite measurable partition $\alpha$ of $X$ at scale $\eta$. 
	Then $\alpha_{F_i}$ is at scale $\eta_{F_i}$ and hence $H_\mu[\eta_{F_i}]\leq H_\mu(\alpha_{F_i})$. This shows 
	\[\sup_{\eta\in \mathbb{U}_X} \operatorname{E}_\mu^{(OW)}[\pi,\eta]
	\, = \, \sup_{\eta\in \mathbb{U}_X} \limsup_{i\in I}\frac{H_\mu[\eta_{F_i}]}{\theta(A_i)}
	\, \leq \, \sup_{\alpha} \limsup_{i\in I}\frac{H_\mu(\alpha_{F_i})}{\theta(A_i)}
	\, = \, \operatorname{E}_{\mu,\mathcal{A},\omega}^{(OW)}(\pi),\]
	where the last supremum is taken over all finite measurable partitions $\alpha$ of $X$. 
	
	To show the reverse inequality, let $\epsilon>0$ and $\alpha$ be a finite measurable partition of $X$. By Lemma \ref{lem:OW:entropylemma1}, there exists an entourage $\eta\in \mathbb{U}_X$ such that, for any finite measurable partition $\gamma$ of $X$ at scale $\eta$, we have $H_\mu(\alpha|\gamma)<\epsilon$. 
	For $i\in I$ we consider a finite measurable partition $\beta$ of $X$ at scale $\eta_{F_i}$. 
	Clearly $\beta_{g^{-1}}:=\{g(B);\, B\in \beta\}$ is at scale $\eta$ for any $g\in F_i$ and hence 
	$H_\mu(\alpha|\beta_{g^{-1}})<\epsilon$. This observation allows to compute
	\begin{align*}
		H_\mu(\alpha_{F_i}) \,
		&\leq \, H_\mu(\beta)+H_\mu(\alpha_{F_i}|\beta)
		\, \leq \, H_\mu(\beta)+H_\mu(\alpha_g|\beta)\\
		&= \, H_\mu(\beta)+H_\mu(\alpha|\beta_{g^{-1}})
		\, \leq \, H_\mu(\beta)+|F_i|\epsilon. 
	\end{align*}
	Taking the infimum over all finite measurable partitions $\beta$ of $X$ at scale $\eta_{F_i}$, it follows that \begin{displaymath}
		H_\mu(\alpha_{F_i}) \, \leq \, H_\mu[\eta_{F_i}]+|F_i|\epsilon \, \leq \, H_\mu[\eta_{A_i}]+|F_i|\epsilon .
	\end{displaymath} Here we used $F_i\subseteq A_i$ for the second inequality. 
	By our choice of $V$ we furthermore know that $(Vg)_{g\in F_i}$ is a disjoint family. 
	Hence
	$\theta({V}A_i)\geq\theta(VF_i) =\sum_{g\in F_i}\theta(Vg)=|F_i|\theta(V)$,	
	and Lemma \ref{lem:vanHovenets} implies that
	\[\limsup_{i\in I}\frac{|F_i|}{\theta(A_i)}=\limsup_{i\in I}\frac{|F_i|}{\theta(VA_i)}\leq \frac{1}{\theta(V)}.\]
	Combining our observations we conclude that
	\begin{align*}
		\limsup_{i\in I}\frac{H_\mu(\alpha_{F_i})}{\theta(A_i)}
		\, \leq \, \limsup_{i\in I}\frac{H_\mu[\eta_{A_i}]}{\theta(A_i)}+\epsilon\limsup_{i\in I}\frac{|F_i|}{\theta(A_i)}
		\, \leq \, \operatorname{E}_\mu^{(OW)}[\pi,\eta] +\frac{\epsilon}{\theta(V)}.		
	\end{align*}
		Since $\epsilon>0$ was arbitrary, this shows that \begin{displaymath}
			\limsup_{i\in I}{H_\mu(\alpha_{F_i})}/{\theta(A_i)} \, \leq \, \operatorname{E}_\mu^{(OW)}[\pi,\eta] \, \leq \, \sup_{\eta\in \mathbb{U}_X} \operatorname{E}_\mu^{(OW)}[\pi,\eta] .
		\end{displaymath} Taking the supremum over all finite measurable partitions $\alpha$ of $X$ now yields the desired statement. 
	\end{proof}
	\begin{remark}
	Note that the previous proof also shows that, for any finite measurable partition $\alpha$ of $X$, there exists a constant $c\in \mathbb{R}$ (independent of the choice of $(A_i)_{i\in I}$ or $\omega$) such that 
	\[\limsup_{i\in I}\frac{H_\mu(\alpha_{A_i\cap\omega})}{\theta(A_i)}\leq c.\] Indeed,  there exists $\eta\in \mathbb{U}_X$ such that $\limsup_{i\in I}{H_\mu(\alpha_{A_i\cap\omega})}/{\theta(A_i)}\leq \operatorname{E}_\mu^{(OW)}[\pi,\eta]$, where the latter is finite as a consequence of the Ornstein-Weiss lemma.  
	
	Furthermore, with similar arguments as above one shows that the limit superior in the formula of Ornstein-Weiss entropy can be replaced by a limit inferior. 
	\end{remark}
		
	Recall that we introduced Delone sets, as we wanted $A_i\cap\omega$ to be finite. 
	It is natural to ask, whether one can replace Delone sets by locally finite and relatively dense sets in our definition of entropy. The next example shows that this is not the case. 
	\begin{example}\label{exa:locallyfinite}
	Let $\mathbb{T} \defeq \mathbb{R}\big/ \mathbb{Z}$ be the circle equipped with the Lebesgue measure $\lambda$. Then $\pi^g(x) \defeq x+g\mod 1$ defines a continuous action of $\mathbb{R}$ on $\mathbb{T}$ with $\operatorname{E}_\lambda^{(OW)}(\pi)=0$.
	Consider the finite measurable partition $\alpha \defeq \{[0,1/2),[1/2,1)\}$ of $\mathbb{T}$. 
	Then  $\omega \defeq \mathbb{Z}\cup\bigcap_{n\in \mathbb{N}}([n,n+1]\cap \{2^{-n}z \mid z\in \mathbb{Z}\})$ is a locally finite and relatively dense set for which
	$\alpha_{[0,n]\cap \omega}$	consists of $2^n$ intervals of equals length. Thus,
	\[\lim_{n\to \infty}\frac{H_\lambda(\alpha_{[0,n]\cap\omega})}{\theta([0,n])} \, = \, \log(2).\]
	\end{example}

\subsection{Ornstein-Weiss topological pressure}
	Let $\pi$ be a continuous action of an amenable unimodular locally compact group $G$ on a compact Hausdorff space $X$. 
	Consider any $f\in C(X)$ and $\eta\in \mathbb{U}_X$. 
	Define $P_f[\eta] \defeq \inf_{\mathcal{U}}P_f(\mathcal{U})$, where the infimum is taken over all finite open covers $\mathcal{U}$ of $X$ at scale $\eta$. 
	We next show that the Ornstein-Weiss lemma can also be applied in this context. 
	
	\begin{proposition}
		Let $\pi$ be a continuous action of an amenable unimodular locally compact group $G$ on a compact Hausdorff space $X$, and let $f\in C(X)$ and $\eta\in \mathbb{U}_X$. For every van Hove net $(A_i)_{i\in I}$ in $G$, the limit
		\[\operatorname{p}_f[\pi,\eta] \, \defeq \, \lim_{i\in I}\frac{P_{f_{A_i}}(\eta_{A_i})}{\theta(A_i)}\]
		exists, is finite, and does not depend on the particular choice of $(A_i)_{i\in I}$.
	\end{proposition}
	\begin{proof}
		It is straightforward to show that, if $f \geq 0$, then $\mathcal{K}(G)\ni A\mapsto P_{f_A}(\eta_A)$ is monotone, right-invariant and subadditive, and then the statement follows from the Ornstein-Weiss lemma. 
	In general, since $X$ is compact, there exists $c\in \mathbb{R}$ with $f+c \geq 0$. Another standard argument shows that
	$P_{(f+c)_A}(\eta_A)=P_{f_A}(\eta_A)+c\theta(A)$ for every compact subset $A\subseteq G$, and hence we obtain the general statement. 		
	\end{proof}
	
	Similar as above we have the following. 
	
	\begin{theorem}\label{the:OW:pressureindependence}
	Let $\pi$ be a continuous action of an amenable unimodular locally compact group $G$ on a compact Hausdorff space $X$ and $f\in C(X)$. For any van Hove net $\mathcal{A}$ and any Delone set $\omega$ in $G$, we have
		\[\operatorname{p}_{f,\mathcal{A},\omega}^{(OW)}(\pi)=\sup_{\eta\in \mathbb{U}_X} \operatorname{p}_f[\pi,\eta].\]
	In particular, $\operatorname{p}_f^{(OW)}(\pi) \defeq \operatorname{p}_{f,\mathcal{A},\omega}^{(OW)}(\pi)$ is independent of the choice of a van Hove net and the choice of a Delone set. 
	\end{theorem}

	\begin{proposition}\label{pro:OW:pressureintermideate}
		Let $\pi$ be a continuous action of an amenable unimodular locally compact group $G$ on a compact Hausdorff space $X$.  
		For any $f\in C(X)$, any van Hove net $(A_i)_{i\in I}$ and any locally finite and relatively dense set $\omega$ in $G$ we have
		\[\sup_{\eta\in \mathbb{U}_X} \operatorname{p}_f[\pi,\eta]
		=\sup_{\eta\in \mathbb{U}_X} \limsup_{i\in I}\frac{P_{f_{A_i}}(\eta_{(A_i\cap \omega)})}{\theta(A_i)}.\]
	The formula remains valid if the limit superior is replaced by a limit inferior. 
	\end{proposition}
	\begin{proof}
	Let $K \subseteq G$ be a compact subset such that $K\omega=G$ and $e_G\in K$.
	Define $F_i \defeq A_i\cap \omega$. 
	From Lemma \ref{lem:vanHovenets} we know that $(KF_i)_{i\in I}$ is a van Hove net in $G$ that satisfies \begin{displaymath}
		\lim_{i\in I}\theta(KF_i)/\theta(A_i) \, = \, 1
	\end{displaymath} and 
	$\lim_{i\in I}\theta((KF_i)\triangle A_i)/\theta(KF_i)=0$. 
	From 
	\begin{align*}
	\left|\frac{P_{f_{KF_i}}(\eta_{KF_i})}{\theta(KF_i)}-\frac{P_{f_{A_i}}(\eta_{KF_i})}{\theta(KF_i)}\right|
	\leq \frac{\|f_{KF_i}-f_{A_i}\|_\infty}{\theta(KF_i)}
	\leq \frac{\theta(KF_i\triangle A_i)}{\theta(KF_i)}\|f\|_\infty
	\end{align*}
	we thus obtain 
	\[\operatorname{p}_f[\pi,\eta]
	=\lim_{i\in I}\frac{P_{f_{KF_i}}(\eta_{KF_i})}{\theta(KF_i)}
	=\lim_{i\in I}\frac{P_{f_{A_i}}(\eta_{KF_i})}{\theta(KF_i)}.\]
	From $(\eta_K)_{F_i}=\eta_{KF_i}$ and $\eta_K\in \mathbb{U}_X$ we thus observe 
	\begin{align*}
	\operatorname{p}_f[\pi,\eta]
	\leq \sup_{\epsilon\in \mathbb{U}_X}\limsup_{i\in I}\frac{P_{f_{A_i}}(\epsilon_{F_i})}{\theta(KF_i)}
	\leq \sup_{\epsilon\in \mathbb{U}_X}\lim_{i\in I}\frac{P_{f_{A_i}}(\epsilon_{KF_i})}{\theta(KF_i)}
	= \sup_{\epsilon\in \mathbb{U}_X}\operatorname{p}_f[\pi,\epsilon].
	\end{align*}
	Taking the supremum over all $\eta\in \mathbb{U}_X$ yields the statement. 
\end{proof}
\begin{proof}[Proof of Theorem \ref{the:OW:pressureindependence}]
	Again we denote $F_i \defeq A_i\cap \omega$. 
	Consider $\eta\in \mathbb{U}_X$ and a finite open cover 
	$\mathcal{U}$ at scale $\eta$. For $i\in I$ we observe that $\mathcal{U}_{F_i}$ is at scale $\eta_{F_i}$ and hence
	\begin{displaymath}	P_{f_{A_i}}[\eta_{F_i}] \, \leq \, P_{f_{A_i}}(\mathcal{U}_{F_i}) . \end{displaymath} 
	From Proposition \ref{pro:OW:pressureintermideate} we thus obtain 
	\begin{align*}
	\operatorname{p}_f^{(OW)}[\pi,\eta]
	=\limsup_{i\in I}\frac{P_{f_{A_i}}[\eta_{F_i}]}{\theta(A_i)}
	\leq \limsup_{i\in I}\frac{P_{f_{A_i}}(\mathcal{U}_{F_i})}{\theta(A_i)}
	\leq \operatorname{p}_f^{(OW)}(\pi). 
	\end{align*}
	As $\eta$ was arbitrary we observe 
	$\sup_{\eta\in \mathbb{U}_X}\operatorname{p}_f^{(OW)}[\pi,\eta]
	\leq  \operatorname{p}_f^{(OW)}(\pi)$. 
	To show the reverse direction let $\mathcal{U}$ be a finite open cover of $X$. 
	Consider a Lebesgue entourage $\eta$ of $\mathcal{U}$. 
	It is straightforward to show that any finite open cover $\mathcal{V}$ at scale $\eta_{F_i}$ is finer than $\mathcal{U}_{F_i}$ and we observe $P_{f_{A_i}}(\mathcal{U}_{F_i})\leq P_{f_{A_i}}[\eta_{F_i}]$ and we conclude the statement of Theorem \ref{the:OW:pressureindependence} from Proposition \ref{pro:OW:pressureintermideate}.  
\end{proof}

\begin{remark} 
	The arguments above allow to conclude that for any locally finite and relatively dense subsets $\omega$ (and any van Hove net $(A_i)_{i\in I}$) we have 
	\[\operatorname{p}_f^{(OW)}(\pi)= \sup_{\mathcal{U}}\limsup_{i\in I}\frac{P_{f_{A_i}}(\mathcal{U}_{A_i\cap \omega})}{\theta(A_i)}.\]
	See Example \ref{exa:locallyfinite} for a locally finite and relatively dense subset that is not a Delone set. 
	\end{remark}
	
	\begin{remark}
	Lemma \ref{lem:nv:Arefining} facilitates a third approach to Ornstein-Weiss topological pressure. 
	Indeed, recall from this lemma that for any finite open cover $\mathcal{U}$ and any compact set $A$ of $G$ there exists a finite open cover $\mathcal{V}$ that $A$-refines $\mathcal{U}$. 
	For $f\in C(X)$ we define 
	$P_f(\mathcal{U},A) \defeq \inf_{\mathcal{V}}P_f(\mathcal{V})$, 
	where the infimum is taken over all such $\mathcal{V}$. 
	A straightforward argument shows that, if $f$ is positive, then $\mathcal{K}(G) \to \mathbb{R}_{\geq 0}, \, A\mapsto P_{f_A}(\mathcal{U},A)$ is monotone, right invariant and subadditive. For general $f\in C(X)$, we apply the Ornstein-Weiss lemma as above to obtain that the following limit exists, is finite and does not depend on the choice of a van Hove net $(A_i)_{i\in I}$ in $G$. We define 
	\[\operatorname{p}^{(OW)}_f(\pi,\mathcal{U}) \, \defeq \, \lim_{i\in I}\frac{P_{f_{A_i}}(\mathcal{U},A_i)}{\theta(A_i)}.\]
	Note that, for every finite set $F\subseteq G$, we have $P_f(\mathcal{U},F)=P_f(\mathcal{U}_F)$. 
	With a similar argument as in the proof of Proposition \ref{pro:OW:pressureintermideate} we observe 
	\[\operatorname{p}^{(OW)}_f(\pi)= \sup_{\mathcal{U}}\operatorname{p}^{(OW)}_f(\pi,\mathcal{U}),\]
	where the supremum is taken over all finite open covers $\mathcal{U}$ of $X$.
	\end{remark}

\subsection{Restricting to uniform lattices}

	With the following lemma we introduce the notion of a \emph{density} $\operatorname{dens}(\omega)$ of a uniform lattice $\omega$ in an amenable unimodular amenable group. This allows to relate averaging along van Hove nets in $G$ to averaging along van Hove nets in the discrete subgroup $\omega$. 

\begin{lemma}\label{lem:vanHovenetinlattice}
		Let $\omega$ be a uniform lattice in an amenable unimodular locally compact group $G$ and $(A_i)_{i\in I}$ be a van Hove net in $G$. Then $(A_i\cap \omega)_{i\in I}$ is a van Hove net in $\omega$ for which 
	\begin{displaymath}
			\operatorname{dens}(\omega) \, \defeq \, \lim_{i\in I} \frac{|A_i\cap \omega|}{\theta(A_i)}
	\end{displaymath} exists. This limit is independent of the choice of the van Hove net $(A_i)_{i\in I}$. 
	\end{lemma}
	\begin{proof}
	Let $C$ be a regular and precompact fundamental domain of $\omega$ with non-empty interior.
	Let $K$ denote the closure of $C$ and $F_i \defeq A_i\cap \omega$. 
	Then $\theta(KF_i)=\theta(K)|F_i|$ and we obtain 
	$\lim_{i\in I}|F_i|/\theta(A_i)=\theta(K)^{-1}\lim_{i\in I}\theta(KF_i)/\theta(A_i)=\theta(K)^{-1}$ from Lemma \ref{lem:vanHovenets}. 
	With a standard argument one obtains that the existence of this limit (for all van Hove nets) implies the independence of the choice of a van Hove net \cite{krieger2010ornstein}. 
	
	In order to show that $(F_i)_{i\in I}$ constitutes a van Hove net in $\omega$, let $F\subseteq G$ be a compact subset and denote by $\partial_F^\omega F_i$ the $F$-boundary of $F_i$ in $\omega$. A straightforward computation shows that
	$C\partial_F^\omega F_i \subseteq \partial_{KF}^G A_i$ and hence
	\begin{align*}
	0 \, &\leq \, \limsup_{i\in I}\frac{|\partial_F^\omega F_i|}{|F_i|} \, = \, \limsup_{i\in I}\frac{\theta(C\partial_F^\omega F_i)}{\theta(CF_i)}\, \leq \, \limsup_{i\in I}\frac{\theta(\partial_{KF}^G A_i)}{\theta(B_i)} \, = \, \lim_{i\in I}\frac{\theta(\partial_{KF}^G A_i)}{\theta(A_i)} \, = \, 0.\qedhere
	\end{align*}
	\end{proof}

	From the following we obtain that Ornstein-Weiss entropy and topological pressure restrict to the classical definitions in the context of actions of discrete amenable groups. 	
	In particular, we obtain that our notion agrees with the notion considered in  \cite{feldman1980rentropy} in the context of actions of $\mathbb{R}^d$.

	\begin{theorem}\label{the:latticerestriction}
	Let $\pi$ be an action of an amenable unimodular locally compact group $G$ and $\phi$ be the restriction of $\pi$ to a uniform lattice $\omega\subseteq G$. Let $K$ be the closure of a regular and precompact fundamental domain of $\omega$
\begin{itemize}
\item[(i)] If $\pi$ is a measure preserving action on a probability space $(X,\mu)$, so is $\phi$ and we have	
		\[\operatorname{E}^{(OW)}_\mu(\pi)=\operatorname{dens}(\omega)\operatorname{E}^{(OW)}_\mu(\phi).\]
		\item[(ii)] If $\pi$ is a continuous action on a compact Hausdorff space $X$, so is $\phi$ and for any $f\in C(X)$ we have		
		\[\operatorname{p}^{(OW)}_f(\pi)
	=\operatorname{dens}(\omega)\operatorname{p}^{(OW)}_{f_K}(\phi).\]
\end{itemize} 
	\end{theorem}
		
	\begin{proof}[Proof of Theorem \ref{the:latticerestriction}]
	We obtain the first formula as a direct consequence of the definition of the Ornstein-Weiss entropy and Lemma \ref{lem:vanHovenetinlattice}. 
	In order to show the second formula we consider the closure $K$ of a regular and precompact fundamental domain $C$ of $\omega$ and a van Hove net $(A_i)_{i\in I}$ in $G$. 
		Let $K$ denote the closure of $C$ and $F_i \defeq A_i\cap \omega$. 
		Now recall from Lemma \ref{lem:vanHovenets} that $(KF_i)_{i\in I}$ is a van Hove net in $G$. From the proof of Lemma \ref{lem:vanHovenetinlattice} we know that $\operatorname{dens}(\omega)=\theta(K)^{-1}$ and from the regularity of $C$ we observe $f_{KF_i}=\sum_{F_i}f_K$. 
		For $\eta\in \mathbb{U}_X$ we compute
	\begin{align*}
	\operatorname{p}^{(OW)}_f[\pi,\eta]
	=\lim_{i\in I}\frac{P_{f_{KF_i}}(\eta_{KF_i})}{\theta(KF_i)}
	=\lim_{i\in I}\frac{P_{\sum_{F_i}f_{K}}((\eta_K)_{F_i})}{\theta(K)|F_i|}
	=\operatorname{dens}(\omega)\operatorname{p}^{(OW)}_{f_K}[\phi,\eta_K]. 
	\end{align*}
	Consider $k\in K$. 
	Since $\eta_k\supseteq \eta_K\in \mathbb{U}_X$ for $\eta\in \mathbb{U}_X$ we observe
	\[
	\operatorname{p}^{(OW)}_{f_K}(\phi)
	=\sup_{\eta\in \mathbb{U}_X}\operatorname{p}^{(OW)}_{f_K}[\phi,\eta_k]
	\leq\sup_{\eta\in \mathbb{U}_X}\operatorname{p}^{(OW)}_{f_K}[\phi,\eta_K]
	\leq \sup_{\epsilon\in \mathbb{U}_X}\operatorname{p}^{(OW)}_{f_K}[\phi,\epsilon]
	=\operatorname{p}^{(OW)}_{f_K}(\phi)
	\]
	The combination of these observations yields
	\[	
	\operatorname{p}^{(OW)}_f(\pi)
	=\sup_{\eta\in \mathbb{U}_X}\operatorname{p}^{(OW)}_f[\pi,\eta]
	=\operatorname{dens}(\omega)\sup_{\eta\in \mathbb{U}_X}\operatorname{p}^{(OW)}_{f_K}[\phi,\eta_K]
	=\operatorname{dens}(\omega)\operatorname{p}^{(OW)}_{f_K}(\phi). \qedhere
	\]
	\end{proof}

\subsection{Goodwyn's theorem for Ornstein-Weiss entropy}

	Our next objective is to show that Ornstein-Weiss entropy satisfies Goodwyn's theorem. 
\begin{theorem}[Goodwyn's theorem for Ornstein-Weiss entropy]
	\label{the:OW:Goodwyns}
		Let $\pi$ be a continuous action of an amenable unimodular locally compact group $G$ on a compact Hausdorff space $X$. 
		For every invariant, regular Borel probability measure $\mu$ on $X$ and every $f\in C(X)$,
	\[\operatorname{E}_\mu^{(OW)}(\pi)+\mu(f) \, \leq \, \operatorname{p}_f^{(OW)}(\pi).\]
	\end{theorem}

\begin{proof}[Proof of Theorem~\ref{the:OW:Goodwyns}]
	Let $(A_i)_{i\in I}$ be a van Hove net and $\omega$ be a Delone set in $G$. Let $V$ be a compact neighbourhood of the identity element in $G$ and consider a $V$-discrete Delone set $\omega$ in $G$. Let $F_i \defeq \omega\cap A_i$ for all $i\in I$. 
		We obtain
	$|F_i|\theta(V)=\theta(VF_i)\leq \theta(VA_i)$ and thus
	\[\limsup_{i\in I}\frac{|F_i|}{\theta(A_i)} \, =\limsup_{i\in I}\frac{|F_i|}{\theta(VA_i)} \, \leq \, \frac{1}{\theta(V)}.\]
	By Lemma \ref{lem:Goodwynsmainlemma}, for every finite measurable partition $\alpha$ of $X$, there exists a finite open cover $\mathcal{U}$ of $X$ such that, for every $i\in I$, 
	\begin{displaymath} H_\mu(\alpha_{F_i})+\mu(f_{A_i}) \, \leq \, \operatorname{P}_{f_{A_i}}(\mathcal{U}_{F_i})+|F_i| \theta(V) \epsilon .\end{displaymath} 
	Since $\mu$ is invariant we know $\mu(f_{A_i})=\mu(f)\theta(A_i)$. We thus obtain 
	\begin{align*}
	\limsup_{i\in I}\frac{H_\mu(\alpha_{F_i})}{\theta(A_i)}+\mu(f)
	&=\lim_{i\in I}\frac{H_\mu(\alpha_{F_i})+\mu(f_{A_i})}{\theta(A_i)}\\
	&\leq \limsup_{i\in I} \frac{\operatorname{P}_{f_{A_i}}(\mathcal{U}_{F_i})}{\theta(A_i)}+\limsup_{i\in I}\frac{|F_i|}{\theta(A_i)}\theta(V)\epsilon\\
	&\leq \operatorname{p}_f^{(OW)}(\pi)+\epsilon.
	\end{align*}
	Taking the supremum over all finite measurable partitions $\alpha$ of $X$, we conclude that
	\begin{displaymath} \operatorname{E}_\mu^{(OW)}(\pi)+\mu(f) \, \leq \, \operatorname{p}_f^{(OW)}(\pi)+\epsilon . \end{displaymath}
	The desired statement now follows, as $\epsilon>0$ was arbitrary. 
	\end{proof}


\section*{Acknowledgments}

We wish to thank Jean Moulin Ollagnier for an extremely helpful and inspiring discussion concerning the proof of the variational principle. Furthermore, the authors wish to express their sincere gratitude towards Maxime Gheysens for providing Example~\ref{example:maxime} and allowing them to include it in the present work. The first author is grateful to Max Planck Institute for Mathematics in Bonn for its hospitality and financial support.

\footnotesize
\bibliographystyle{alpha}
\bibliography{referencesEntropy}


\end{document}